\documentclass[11pt]{amsart}

\usepackage[paper=a4paper,left=30mm,right=20mm,top=25mm,bottom=30mm]{geometry}

\usepackage{enumerate,url,amssymb, mathrsfs, mathabx, esint, xcolor}
\usepackage{graphicx, comment}
\usepackage[hidelinks, bookmarksdepth=3]{hyperref}
\usepackage{tikz-cd}

\usetikzlibrary{decorations.markings}
\tikzset{degil/.style={
            decoration={markings,
            mark= at position 0.5 with {
                  \node[transform shape] (tempnode) {$\bigtimes$};
                  }
              },
              postaction={decorate}
}
}

\newtheorem{theorem}{Theorem}[section]
\newtheorem{lemma}[theorem]{Lemma}
\newtheorem*{lemma*}{Lemma}
\newtheorem{proposition}[theorem]{Proposition}
\newtheorem{corollary}[theorem]{Corollary}

\theoremstyle{definition}
\newtheorem{definition}[theorem]{Definition}

\theoremstyle{remark}
\newtheorem{remark}[theorem]{Remark}

\numberwithin{equation}{section}

\newcommand{\A}{\mathbb{A}}
\newcommand{\C}{\mathbb{C}}
\newcommand{\D}{\partial}
\newcommand{\DD}{\mathbb{D}}

\newcommand{\R}{\mathbb{R}}

\DeclareMathOperator{\dist}{dist}

\DeclareMathOperator{\loc}{loc}

\def\XXint#1#2#3{{\setbox0=\hbox{$#1{#2#3}{\int}$}
\vcenter{\hbox{$#2#3$}}\kern-.5\wd0}}

\def\le{\leqslant}

\renewcommand{\Re}{\operatorname{Re}}

\def \D{\textup{D}}

\def \tp{\textup}

\def \wstar{\overset{\ast}{\rightharpoonup}}

\newcommand{\Rn}{\mathbb{R}^{n}}
\newcommand{\Leb}{\mathscr{L}}
\newcommand{\N}{\mathbb{N}}
\newcommand{\M}{\mathbb{R}^{2 \times 2}}
\newcommand{\WW}{\mathrm{W}}
\newcommand{\LL}{\mathrm{L}}
\newcommand{\CC}{\mathrm{C}}

\newcommand{\dd}{\mathrm{d}}

\newcommand{\B}{\mathbf{B}}
\newcommand{\weak}{\rightharpoonup}
\newcommand{\weakstar}{\overset{\ast}{\rightharpoonup}}

\newcommand{\restrict}{\begin{picture}(10,8)\put(2,0){\line(0,1){7}}\put(1.8,0){\line(1,0){7}}\end{picture}}

\definecolor{green}{RGB}{40,160,90}
\definecolor{orange}{RGB}{180,80,0}

\author{K.~Astala}
\address{Department of Mathematics and Statistics, University of Helsinki, Finland}
\email{kari.astala@helsinki.fi}

\author{D.~Faraco}
\address{Departamento de Matematicas, Universidad Autonoma de Madrid,  Spain and ICMAT CSIC, Madrid, Spain}
\email{daniel.faraco@uam.es}

\author{A.~Guerra}
\address{Institute for Theoretical Studies ETH-ITS, Z\"urich, Switzerland}
\email{andre.guerra@eth-its.ethz.ch}

\author{A.~Koski}
\address{Department of Mathematics and Systems Analysis, Aalto University, Espoo, Finland}
\email{aleksis.koski@aalto.fi}

\author{J.~Kristensen}
\address{Mathematical Institute, University of Oxford, United Kingdom}
\email{kristens@maths.ox.ac.uk}

\begin{document}
\baselineskip5mm
\title[Lower semicontinuity  ]{Lower semicontinuity, Stoilow factorization and principal maps }

\maketitle

\begin{center}
  \emph{Dedicated to Vladimir \v{S}ver\'ak on the occasion of his 65th birthday}
\end{center}
\begin{abstract}  
We consider a strengthening of the usual quasiconvexity condition of Morrey in two dimensions, which allows us to prove lower semicontinuity for functionals which
are unbounded as the determinant vanishes. This notion, that we call \textit{principal quasiconvexity}, arose from the planar theory of quasiconformal mappings
and mappings of finite distortion. We compare it with other quasiconvexity conditions that have appeared in the literature and provide a number of concrete
examples of principally quasiconvex functionals that are not polyconvex. The Stoilow factorization, that in the context of maps of integrable distortion was
developed by Iwaniec and \v{S}ver\'ak, plays a prominent role in our approach.
\end{abstract}

\section{Introduction}
Given a real $n \times n$ matrix $A$, its (outer) distortion is defined by
$$
K_A \equiv \left\{
\begin{array}{cl}
  |A|^{n}/ \det A & \mbox{ if } \det A > 0,\\
  1 & \mbox{ if } A = 0,\\
  +\infty & \mbox{ if } A \neq 0 \mbox{ and } \det A \leq 0,
\end{array}
\right.
$$
where $| A | \equiv \max_{x \in \mathbb{S}^{n-1}} |Ax|$ is the operator norm and $|Ax|$ is the usual euclidean norm of $Ax \in \Rn$. 
When $\Omega$ is an open, connected, non-empty subset of $\Rn$ (henceforth called a domain) and $u \colon \Omega \to \Rn$ is locally of Sobolev class $\WW^{1,1}$
in $\Omega$, denoted $u \in \WW^{1,1}_{\loc}( \Omega , \Rn )$, its \textit{distortion function} $K_{u}(x) \equiv K_{\D u(x)}$ is an $\Leb^n$-almost everywhere
defined extended real-valued function on $\Omega$.
If the distortion $K_{u} < +\infty$ almost everywhere and the \textit{Jacobian} $J_u \equiv J_{\D u} \equiv \det \D u \in \LL^1_{\loc}( \Omega )$, then $u$ is called
a \textit{mapping of finite distortion}, and if furthermore $K_{u}  \in \LL^{\infty}( \Omega )$, then $u$ is called a \textit{weakly quasiregular map}.
A striking result of Reshetnyak states that a \textit{quasiregular map}, so a map $u \in \WW^{1,n}_{\loc}( \Omega , \Rn )$ with $K_u \in \LL^{\infty}( \Omega )$,
is either constant or admits a representative that is continuous, open and discrete.
In the two-dimensional case, $n=2$, that is our focus here, this result is a consequence of the so-called \textit{Stoilow factorization}. When $\Omega$ is a domain in
$\C$ it asserts that a non-constant map $f \colon \Omega  \to \C$ is continuous, open and discrete if and only if it admits the factorization
$f =H \circ F$, where $F \colon \Omega \to \C$ is a homeomorphism and $H \colon F( \Omega ) \to \C$ is holomorphic (see \cite{Stoilow28}, and \cite{LuistoPankka20}
for a modern exposition). 
In \cite{IwaniecSverak93} Tadeusz Iwaniec and Vladimir \v{S}ver\'ak brought together ideas from Geometric Function Theory and Nonlinear Elasticity proving,
again in two dimensions, that if $u \in \WW^{1,2}_{\loc}( \Omega ) \equiv \WW^{1,2}_{\loc}( \Omega , \C )$ is not a.e.~constant and has distortion function
$K_u  \in \LL^{1}( \Omega)$, then (the precise representative of) $u$ is continuous, open and discrete (see also  \cite{AIMbook} for various refinements).
The aim of this note is to give a systematic presentation of the notion of principal quasiconvexity considered first in \cite{AFGKKpreprint}. In particular we aim to
explain how it, in conjunction with Stoilow factorization and its extension by Iwaniec and \v{S}ver\'ak, allows to prove lower semicontinuity results in
situations where the energy functionals are not polyconvex and do not satisfy the standard growth assumptions. Relevant examples include the typical variational
models for isotropic planar hyperelasticity and the variational problems arising in geometric function theory.  

We briefly review some background results from the calculus of variations before elaborating further on the geometric function theory concepts, such as
Stoilow factorization, the distortion function and the principal maps. 

First recall \cite{Morrey,BM1} that when $\Omega$ is a bounded domain in $\Rn$ and the exponent $p \geq 1$, then the $\WW^{1,p}$-quasiconvexity of a functional
$\mathbf{E} \colon \R^{n \times n} \to \R$, meaning that 
$$ 
 {\bf E}(A)\leq \fint_{\Omega} {\bf E} (\D u(x)) \, \dd x \quad  \; {\rm for\, all} \; u \in A + \WW^{1,p}_0( \Omega , \Rn ) \; {\rm and} \;  A \in \M, 
$$ 
is equivalent to the sequential weak lower semicontinuity on $\WW^{1,p}( \Omega , \Rn )$ (briefly $\WW^{1,p}$-swlsc) of the corresponding variational integral
\begin{equation}\label{energy}
\mathscr{E}[u] \equiv \int_{\Omega} \! {\bf E}(\D u) \, \dd x,
\end{equation}
provided that ${\bf E}$ is non-negative and in addition has standard $p$-growth:
\begin{equation}\label{pgrowth}
| \mathbf{E}(A)| \leq c\bigl( |A|^{p}+1 \bigr) \quad \forall \, A \in \R^{n \times n} .
\end{equation}
It is easy to see that this equivalence persists if the functional $\mathbf{E}$ satisfies \eqref{pgrowth} but, instead of being non-negative, satisfies the weaker
condition
\begin{equation}\label{equi-neg}
\liminf_{|A| \to +\infty} \frac{\mathbf{E}(A)}{|A|^{p}} \geq 0.
\end{equation}
However, when a $\WW^{1,p}$-quasiconvex functional $\mathbf{E}$ has only the standard $p$-growth \eqref{pgrowth} for an exponent $p>1$, then the variational integral
\eqref{energy} can fail to be $\WW^{1,p}$-swlsc, due to possible concentration effects at the boundary \cite{Morrey,BM1,Marcellini1,AcerbiFusco,FoMa}. For instance, when
$n=p=2$ an example of this with $\mathbf{E}(A) = \det A$, $A \in \M$, can be found in \cite{Dacorogna2007}.
On the other hand, as observed by Meyers \cite{Meyers}, the $\WW^{1,p}$-swslc result persists provided the variational integral $\mathscr{E}$ is restricted to Dirichlet classes:
if $g \in \WW^{1,p}( \Rn , \Rn )$, then for a functional $\mathbf{E}$ satisfying
\eqref{pgrowth},  the variational integral $\mathscr{E}$ is $\WW^{1,p}$-swlsc on $g+\WW^{1,p}_{0}( \Omega , \Rn )$  if and only if $\mathbf{E}$ is $\WW^{1,p}$-quasiconvex.
We refer to \cite{BenKru} for further discussions and to \cite{ChJK} for some related results also linking quasiconvexity to coercivity of $\mathscr{E}$
in Dirichlet classes.

In the absence of the standard $p$-growth condition \eqref{pgrowth} the problem of $\WW^{1,p}$-swlsc is much harder, little can be said in the general case and most
results in the literature require some additional assumptions for their validity. In fact already the definition of the variational integral $\mathscr{E}[u]$ is unclear
at this level of generality. When ${\bf E}$ is allowed to be valued in $\R \cup \{ +\infty \}$, or perhaps even in $\R \cup \{ \pm \infty \}$, then we have issues
already for smooth and compactly supported maps $u$. The approach we shall follow here is in a sense the naive one, namely to use \eqref{energy}, but with the integral
interpreted as an upper Lebesgue integral. We often refer to this as the pointwise definition of the variational integral. The other approach to the definition
of the variational integral is to define it via a relaxation process, which in many situations is more natural and often means that the ensuing variational integral
has better properties. This approach is often called the Lebesgue-Serrin-Marcellini definition and in the present context goes back to \cite{Marcellini2}.
If we use the pointwise definition of the variational integral \eqref{energy}, then both the $\WW^{1,p}$-quasiconvexity and the rank-one convexity conditions remain
necessary for $\WW^{1,p}$-swlsc, see \cite{BM1}. However, neither is sufficient. The rank-one convexity does not even imply $\WW^{1,\infty}$-quasiconvexity, regardless
of growth conditions, by \v{S}ver\'{a}k's counter example \cite{Sverak92,Sverak95} except if the mapping is into the target $\R^2$, where the problem remains open, see also \cite{Grabovsky}.
In this connection we refer to \cite{FaracoSzekelyhidi08} and \cite{PSv} for the particularity of $\M$ matrices for this problem and also some partial results.
The sufficiency of $\WW^{1,p}$-quasiconvexity for $\WW^{1,p}$-swlsc of $\mathscr{E}$ on the Dirichlet class $g+\WW^{1,p}_{0}( \Omega , \Rn )$ is an open problem when the
functional ${\bf E}$ is assumed continuous. Easy counterexamples exist if the functional $\mathbf{E}$ is allowed to be discontinuous and extended real-valued, see \cite{BM1}. 
In this situation one introduces a strengthening of the $\WW^{1,p}$-quasiconvexity condition, the so-called {\it closed $\WW^{1,p}$-quasiconvexity} \cite{PedregalBook97}.
It is a sufficient condition for $\WW^{1,p}$-swlsc of $\mathscr{E}$ on $\WW^{1,p}( \Omega , \Rn )$, provided the functional $\mathbf{E}$ satisfies also \eqref{equi-neg}.
However, the necessity becomes unclear, and in fact, when $\mathbf{E}$ is allowed to be lower semicontinuous and extended real-valued, the necessity of closed
$\WW^{1,p}$-quasiconvexity fails, see \cite{JK}. 

These issues become particularly relevant in dealing with variational problems where the effective domain of the energy functional is contained in the set of
  matrices with finite distortion $\{ K_{A} < +\infty \}$. We denote
\begin{equation}\label{det.pos}
\R^{n\times n}_{+}\equiv \R^{n\times n}\cap \{\det >0\} ,
\end{equation}
and record that $\{ K_{A} < +\infty \} = \R^{n \times n}_{+} \cup \{ 0 \}$. Such functionals arise naturally in, for instance, elasticity 
\cite{Ball1,Ball1977,Ball1981} and in geometric function theory \cite{AIMbook}, where a common example is the classical second invariant   
\begin{equation}\label{second invariant}
\mathbf{I_{2}}(A) = K_A  + \frac{1}{K_A }, \qquad A \in \R^{2\times 2}_{+} \cup \{ 0 \},
\end{equation}
which is discussed at length in \cite[Chapt.~21]{AIMbook} and in \cite{AIMO05,MartinYao24}.
However, for applications to elasticity in particular, two central aspects of the theory in this context remain somewhat obscure.
Firstly, it is not clear how ${\bf E}(A)$ should diverge when $\det(A)$ tends to $0$, in order to have a feasible theory for the
corresponding variational integral $\mathscr{E}$.
Secondly, it seems desirable to rule out pathological deformations (for example homeomorphisms whose determinant is zero a.e or sequences of maps which
converge to such maps \cite{HK, FMO18}).  Thus it is not obvious what are the appropriate function spaces in which minimisers should be sought,
and unfortunately it appears that experiments do not help to clarify these issues partly because we are close to a regime where the elasticity theory breaks down.
Let us emphasize that variational problems on the classes of $\WW^{1,p}_{\loc}$ homeomorphisms, lead for $p \geq 2$ naturally to the so-called topologically monotone maps introduced
by Morrey, see \cite{AIMbook} and \cite{IO09,IO16}, whereas for $p<2$ a more complicated condition, the no crossing condition, introduced by De Philippis and Pratelli in
\cite{DePPra20} becomes relevant (see also \cite{CaKaRa21} for the case $p=1$).
For these reasons, and since the general structure of quasiconvex functionals is far from clear, the functionals defined on $\R^{n\times n}_{+}$ that have been successfully
used in the elasticity  literature, have mostly been polyconvex.  See for example, \cite{Ball1981,Ball2002,Ciarlet,Cartesian} and particularly \cite{Ball1,Ball1977,MuSp95,HeMC} for the
setting and success of polyconvexity.  

Similar obstacles were faced in our recent study \cite{AFGKKpreprint} of the local Burkholder functional $\mathbf{B}_{K}^{\loc}$, see \eqref{burkholder}
and \eqref{localburk} below for the definitions. This functional is real-valued on the $K$-quasiconformal well
\begin{equation}\label{quasiwell}
Q_{2}(K) \equiv  \bigl\{ A \in \M : \, K_A \leq K \bigr\}, 
\end{equation}
while outside the well it is $+\infty$, $\mathbf{B}_{K}^{\loc}(A) = +\infty$ when $A \notin Q_{2}(K)$.
Because $\B_{K}^{\loc}$ is non-positive and $p$-homogeneous on $Q_{2}(K)$ it is easy to see that the closed $\WW^{1,p}$-quasiconvexity of $\mathbf{B}_{K}^{\loc}$ is a necessary condition
for $\WW^{1,p}$-swlsc of the associated variational integral. On the other hand, the sufficiency is now non-obvious (assuming the critical case where $K=p/(p-2)$). This is exactly the
opposite situation of what happens for non-negative extended real-valued functionals. 

For later reference we finish this subsection by recalling a rather pleasant property of $\WW^{1,1}$-homeomorphisms with integrable distortion: If $f \colon \Omega \to \C$ is
a homeomorphism (onto $f( \Omega )$) of class $\WW^{1,1}_{\loc}( \Omega )$ with $K_{f} \in \LL^{1}( \Omega )$, then its inverse $f^{-1} \colon f( \Omega ) \to \Omega$ is of
class $\WW^{1,2}( f( \Omega ))$ and 
\begin{equation}\label{l1distort}
\int_{\Omega} \! K_f(z) \, \dd m(z) = \int_{f(\Omega)} \! | \D f^{-1}(w)|^2 \, \dd m(w) .
\end{equation}
See \cite{AIMbook}, \cite{HK}. 

\subsection{Principal maps} In order to address some of the issues described above we next turn to the role of Stoilow factorization in the setting of lower semicontinuity.
Here the key notion is that of a principal map, which has for a long time been implicitly used  in geometric complex analysis, for instance in
connection with planar quasiconformal mappings and the  Beltrami equations; explicitly the term is used in \cite{AIMbook}. We require a slightly more
flexible definition:

\begin{definition}\label{def:principal}
A \textit{principal map} is an orientation preserving homeomorphism  $f \colon \C \to \C$ of class $\WW^{1,1}_{\loc}(\C)$, with integrable distortion $K_f$ on the open
unit disk $\DD$, which is conformal outside the closed unit disc $\overline{\DD}$ with the Laurent expansion
$$
f(z) = b_0 z + \sum_{j=1}^{\infty}\frac{b_{j}}{z^j} \mbox{ for } |z|>1,
$$
where $|b_1| < |b_0|$.
\end{definition}
In connection with the principal map $f$ we often write
$$
A_{f}(z) \equiv b_{0}z+b_{1}\bar{z} \mbox{ for $z \in \C$, and } \phi_{f}(z) \equiv \sum_{j=2}^\infty \frac {b_j}{z^j} \mbox{ for } |z| \geq 1.
$$
Here the linear map $A_{f} \in \R^{2\times 2}$ (that we identify with its matrix) has in particular positive determinant $|b_{0}|^{2}-|b_{1}|^2$,
and the function $\phi_{f}$ extends  continuously to all of $\overline{\C}$, defining a holomorphic map on $\C\setminus \overline{\DD}$.
Since $1/z=\bar z$ for $|z|=1$, we hereby have $f(z)=A_f (z)+\phi_{f}(z)$ for $|z|=1$.
We  also record that
\begin{equation}\label{eq:centermass}
\fint_{\DD} \! \D f(z) \, \dd m(z) =  A_f.
\end{equation}
Notice that it is standard to normalize principal maps to have $b_0=1$. Here we shall refer to such maps as \textit{normalized principal maps}.
It is useful to recall that the class of all normalized $\WW^{1,2}$-principal maps forms a normal family (see \cite[Theorems 20.1.6 and  20.2.3]{AIMbook});
in fact, any $\WW^{1,2}_{\loc}$-bounded sequence of such normalized principal maps admits a subsequence that converges uniformly on $\C$ to a normalized
principal map. On the other hand, normalized principal maps with the mere $\WW^{1,1}_{\loc}$-regularity {\it do not} form a normal family; a sequence of
such Sobolev homeomorphisms can have a BV-limit (see \cite[Chapt.~21]{AIMbook}).

We also recall that via  \eqref{l1distort} the inverse of a principal map lies in $\WW^{1,2}_{\loc}( \C )$.
If the principal map itself lies in $\WW^{1,2}_{\loc}$, then the inverse has likewise an integrable distortion.

Our next goal is the definition of $\WW^{1,p}$-principal quasiconvexity, which is instrumental for the lower semicontinuity results in this paper.
It can be viewed as a refinement, adapted to the setting described above, of the notions of $\WW^{1,p}$-quasiconvexity that was introduced by
Ball and Murat in \cite{BM1} and of closed $\WW^{1,p}$-quasiconvexity that was introduced by Pedregal in \cite{Pedregal}. The condition is motivated
by the Stoilow factorization mentioned above.

For one of our key examples it is important that we allow the considered integrands to be signed, so before giving our definition of principal quasiconvexity 
let us for comparison and precision also review the corresponding definitions of $\WW^{1,p}$-quasiconvexity and closed $\WW^{1,p}$-quasiconvexity in this context.

A lower semicontinuous functional ${\bf E} \colon \M_{+} \to \R \cup \{ + \infty \}$ is said to be $\WW^{1,p}$-quasiconvex at $A \in \M_{+}$ provided
\begin{equation}\label{w1p-qc}
{\bf E}(A) \leq \fint_{\DD}^{\ast} \! {\bf E} (\D f(z)) \, \dd m(z)
\end{equation}
holds for all $f \in A+\WW^{1,p}_{0}( \DD )$\footnote{The asterix on the integral signifies that the integral is intended as an upper Lebesgue integral, meaning that when both
${\bf E}( \D f)^{+}$ and ${\bf E}( \D f)^{-}$ integrate to $+\infty$ over $\DD$, then the integral is taken to be $+\infty$. The interested reader can verify that this
is equivalent to requiring, for each $k \in \N$, the validity of the inequality \eqref{w1p-qc} when ${\bf E}$ is replaced by $\max\{ {\bf E},-k \}$.}.
If this Jensen inequality holds for all $A \in \M_{+}$, then ${\bf E}$ is said to be
$\WW^{1,p}$-quasiconvex.

As observed in \cite{BM1} this condition depends on the exponent $p$ and it refines Morrey's original definition of quasiconvexity in \cite{Morrey} that
in this terminology corresponds to $\WW^{1,\infty}$-quasiconvexity. The definition is arguably natural, but it still has some deficiencies. For instance, as
observed in \cite{BM1}, an extended real-valued integrand need not be rank-one convex.  The next definition strengthens the notion of $\WW^{1,p}$-quasiconvexity
and does always imply rank-one convexity. 

A lower semicontinuous integrand ${\bf E} \colon \M_{+} \to \R \cup \{ +\infty \}$ is closed $\WW^{1,p}$-quasiconvex at $A \in \M_{+}$ provided
\begin{equation}\label{closed-w1p-qc}
{\bf E}(A) \leq \int_{\M_{+}}^{\ast} \! {\bf E} \, \dd \nu
\end{equation}
holds for all homogeneous $\WW^{1,p}$ gradient Young measures $\nu$ with centre of mass $A$. If it holds for all $A \in \M_{+}$, then ${\bf E}$
is said to be closed $\WW^{1,p}$-quasiconvex.  We refer for the definitions and basic properties of Young measures to the monographs  \cite{MullerBook99,PedregalBook97,Rindler}.
In particular, we will use the notation and terminology of gradient Young measures as described in our paper \cite{AFGKKpreprint}.
In fact, closed quasiconvexity provides an abstract solution to a closely related lower
semicontinuity problem when the considered integrands are assumed bounded from below, see \cite{JK}. We refer to Section \ref{sec:examples} and also 
\cite{AFGKKpreprint} for further discussion of these quasiconvexity notions. The main drawback is that it is very difficult to verify whether a
$\WW^{1,p}$-quasiconvex integrand is closed quasiconvex unless we have control on the generating sequence of the gradient Young measures.  

Now we are ready for our definition of principal quasiconvexity. 

\begin{definition}\label{def:pqc}
A lower semicontinuous integrand ${\bf E} \colon \M_{+} \to \R \cup \{ +\infty \}$ is $\WW^{1,p}$-\textit{principal quasiconvex} at $A \in \M_{+}$ if the Jensen
inequality 
\begin{equation} \label{prinqc}
{\bf E}(A)\leq \fint_{\DD}^{\ast} \! {\bf E} (\D f(z)) \, \dd m(z)
\end{equation}
holds for all principal maps $f \colon \C \to \C$ of class $\WW^{1,p}_{\loc}( \C )$ with $\fint_{\DD} \! \D f(z) \, \dd m(z) = A$.
We say that ${\bf E}$ is $\WW^{1,p}$-principal quasiconvex if the Jensen inequality \eqref{prinqc} holds for all 
$A \in \M_{+}$ and for the corresponding $\WW^{1,p}_{\loc}$-principal maps $f$. 
\end{definition} 
We have stated an abstract definition depending on $p$ parallel to that of $\WW^{1,p}$-quasiconvexity and closed $\WW^{1,p}$-quasiconvexity, however our examples
all are $\WW^{1,1}$-principal quasiconvex (Notice that we are testing the Jensen inequality with gradients of homeomorphism). 
We also remark that allowing for extended real-valued integrands is a matter of
convenience, but it also means that it is easy to construct pathological examples. For example, a $\WW^{1,1}$-principal quasiconvex integrand need
not be rank-one convex: take the integrand that is zero at the two matrices $\mathrm{diag }(1,1)$, $\mathrm{diag }(1,2)$ and $+\infty$ elsewhere.
It is unclear if similar examples exist when the functional is required to be continuous.
  
Often in our lower semicontinuity theorems it is convenient that the integrands ${\bf E}$ are defined on the full matrix space $\M$ and not merely on $\M_{+}$.
We achieve this by simply declaring
\begin{equation}\label{convention}
{\bf E}(A) = +\infty \, \mbox{ on } \, \M \setminus \bigl( \M_{+} \cup \{ 0 \} \bigr)
\end{equation}
and throughout the entire paper this will be our standing assumption for the considered integrands that might not be explicitly defined beyond $\M_{+}$.
The reader will notice that we did not specify any value for ${\bf E}$ at $0 \in \M$. This is on purpose because the $0$ matrix is in this context special and
it is convenient to have an ad hoc approach here. Sometimes we extend ${\bf E}$ to $0$ by lower semicontinuity and sometimes we declare that it
is $+\infty$ or some other value there.

Next, we compare principal quasiconvexity with the classical notions of semiconvexity in the vectorial calculus of variations. Firstly,
it is immediate from Jensen's inequality  for convex functions and  \eqref{eq:centermass} that any lower semicontinuous convex functional on $\M$ restricts
to a $\WW^{1,1}$-principal quasiconvex functional on $\M_{+}$. Next, 
notice that  if  $f \in \CC^{1}( \DD ) \cap \WW^{1,\infty}(\DD)$ with $\det \D f>0$, and for some linear $A \in  \M_{+}$ we have $f \in A+\WW^{1,\infty}_{0}(\DD )$,
so that $f$ is a classical test function for the $\WW^{1,\infty}$-quasiconvexity inequality, then $f$ can be extended as principal map off the unit disc. Indeed, if
$A(z) = b_{0}z + b_{1} \bar{z}$, then the  map $f^{\rm prin}(z) \equiv f(z) \mathbf{1}_\DD+ \bigl(b_0 z+b_{1}/z \bigr)\mathbf{1}_{\C \setminus \DD}$ is principal in the sense
of Definition~\ref{def:principal}. Therefore $\WW^{1,\infty}$-principal quasiconvexity implies the standard quasiconvexity inequality for sufficiently smooth maps.
A standard argument then yields that a $\WW^{1,\infty}$-principal quasiconvex functional is locally rank-one convex on the interior of its effective domain and
therefore that it is locally Lipschitz there (see \cite{Dacorogna2007}). 

On the other hand, as we shall see in Section 2 below, the classical area formula implies that minus the determinant,
$- \det(A) $, is $\WW^{1,1}$-principal quasiconvex (although not $\WW^{1.p}$-quasiconvex for $1 \le p<2$) , while the determinant itself, $\det (A)$, is not.
This simple example shows, in particular, that polyconvexity, and hence also quasiconvexity, does not imply principal quasiconvexity in general.
For general functionals on $\M_{+}$ we present an additional local  condition, see~\eqref{sh}, which together with standard quasiconvexity
implies principal quasiconvexity, see Proposition~\ref{superharmonic}.  

In addition,  the class of boundary values of all principal maps, i.e. of all the maps $f$ for which \eqref{prinqc} applies, is much larger than
just the linear ones. In fact, if $\Omega$ is a Jordan domain in $\overline \C$ with $\infty \in \Omega$, take a conformal map
$\Psi \colon\C \setminus \overline{\DD} \to \Omega$ with $\Psi(\infty) = \infty$. Then if the boundary $\partial \Omega$ is sufficiently regular, we can extend
$\Psi$ to a principal map of  $\overline \C$.  In particular, we see that the set of boundary values of principal maps on $\partial \DD$ coincides with the set of
boundary values on $\partial \DD$ of all such Riemann maps.
\smallskip

Let us then  informally explain  why principal quasiconvexity is relevant for establishing lower semicontinuity of variational integrals. 
The proof of  Morrey \cite{Morrey} for the lower semicontinuity of functionals with standard quasiconvex integrands goes
basically  as follows. Firstly, given a weakly converging  sequence $( \psi_j )$ of Sobolev functions, one 
localises the sequence to achieve convergence to a linear map. Second, one replaces this localized sequence with one of affine boundary values,
$u_j-Az \in \WW_{0}^{1,2}(\Omega)$, via a carefully chosen cutoff. And third, one applies directly the definition of quasiconvexity to each such $u_j$.
However, such a strategy breaks down for functionals that are  defined and finite only on $\R^{2\times 2}_{+}$, simply because the
cut-off modification might give negative determinants on sets of positive measure. 

We must therefore approach the lower semicontinuity problem via other methods, for example  gradient Young measures and the related closed quasiconvexity.
However, to deal with gradient Young measures, and investigate whether quasiconvex functionals are closed quasiconvex
in practise one needs information of the generating sequence of the Young measure. It is here that the principal quasiconvexity shows its usefulness.
For instance, in \cite{AstalaFaraco02} one first shows that
weakly converging sequences of quasiregular mappings generate gradient Young measures that are supported on the quasiconformal wells $Q_2(K)$,
as defined in \eqref{quasiwell}. And second, thanks to Stoilow factorization, the properties of those Young measures can be analysed with the
help of principal maps, via the next result from  \cite{AFGKKpreprint}. After these observations,  the definition of principal quasiconvexity
seems the natural one. 

\begin{proposition}\label{thm:YMqc}
Consider a homogeneous $\WW^{1,s}$ gradient Young measure $\nu$ supported on the $K$-quasiconformal well $Q_{2}(K)$, 
where the exponent $2K/(K+1) <  s < 2K/(K-1)$.
Then there is a sequence of $K$-quasiconformal principal maps which generates in $\DD$ the Young measure $\nu$.
\end{proposition}
For mappings of bounded distortion, aka the quasiconformal maps, the classical Stoilow factorization is among the main tools in the above result.
However, looking for similar methods for mappings of finite or integrable distortion, it is the Iwaniec--\v{S}ver\'{a}k version of Stoilow factorization
that provides us the following important improvement, see \cite[Theorem 4.4]{AFGKKpreprint}.

\begin{proposition}\label{thm:intdistort}
Let $\nu$ be a homogeneous Young measure generated by $( \D \psi_j )$, where $( \psi_j )$ is a sequence of homeomorphisms that
is bounded in $\WW^{1,2}(\DD )$ and such that $\sup_{j} \int_{\DD} \! K^{q}_{\psi_j} \,\dd m(z) < +\infty$ for some $q>1$. 
Then there is a sequence of homeomorphisms $f_j\colon \C\to \C$ such that:
\begin{enumerate}
\item $f_j$ are principal maps;
\item $(f_j)$ is bounded in $\WW^{1,2}_{\loc}(\DD )$ and $( \D f_{j} )$ generates $\nu$;
\item $\psi_{j} = h_{j} \circ f_{j}$ for some conformal maps $h_{j} \colon f_{j}(\DD )\to \psi_{j}(\DD )$.
\end{enumerate}
\end{proposition}
We recently proved \cite{AFGKKtoappear} a more general version of this theorem that applies to sequences of maps $( \psi_j )$ with distortions
bounded merely in $\LL^1$. In particular we emphasize that the maps $\psi_j$ there need not necessarily be homeomorphisms.
The proof has some new ideas compared to \cite{AFGKKpreprint} and will appear elsewhere.
      
In order to apply Theorem \ref{thm:intdistort} for a functional ${\bf E}$, we need that the sequence $\bigl( {\bf E}(\D f_{j}) \bigr)$ is equi-integrable.
Thus the proofs of results concerning such general functionals, considered in Theorems \ref{thm:lscPQ} and \ref{asfinite-lsc1} with Corollary \ref{existence},
split naturally into two propositions. The first states that the growth conditions, such as given in  \eqref{eq:growthconditions} below, together
  with a higher integrability of the distortion functions, guarantee that the sequence $\bigl( {\bf E}(\D f_{j}) \bigr)$ is equi-integrable.
This rules out concentration effects. 

The second proposition, needed in Theorem \ref{thm:lscPQ},  takes care of the oscillation effects and yields lower
semicontinuity once equi-integrability is available. 
As a matter of fact, in \cite{AFGKKpreprint} a similar issue arose for the important Burkholder functional itself. Fortunately the scope of
principal quasiconvexity and Stoilow factorization goes beyond the basic oscillation effects. Indeed, our third result, 
Theorem \ref{swlscBurk}  shows that in the case of the Burkholder functional, principal quasiconvexity allows us to use a blow-up technique to prove
lower semicontinuity in the borderline case, where equi-integrability cannot be assumed.  

We next describe the results. 

\subsection{Main results}\label{results}

To illustrate the uses of principal quasiconvexity we next present three lower semicontinuity theorems, 
and start with the classical Burkholder functional $\B_p \colon \M \to \R,$ which for exponents $p \geq 2$ is given by 
\begin{equation}\label{burkholder}
\B_p (A) \equiv  \Biggl( \bigl(\frac{p}{2} - 1 \bigr) |A|^2 - \frac{p}{2}\det A \Biggr) |A|^{p-2}, \qquad A \in \M. 
\end{equation}
We note that $\B_{2}(A) = -\det(A)$, while for $p>2$,  in terms of the quasiconformal well \eqref{quasiwell},
we have $\B_p(A) \leq 0$ if and only if $A \in Q_{2}\bigl( p/(p-2) \bigr)$. 
Thus, fixing $K \equiv p/(p-2)$, or equivalently, $p = 2K/(K-1)$, one is led \cite{AFGKKpreprint} to consider 
the \textit{local Burkholder functional} defined as
\begin{equation}\label{localburk}
\mathbf{B}_{K}^{\loc}(A) \equiv \left\{
\begin{array}{ll}
  \mathbf{B}_p(A) & \mbox{ if } A \in Q_{2}(K)\\
  +\infty & \mbox{ if } A \in \M \setminus Q_{2}(K).
\end{array}
\right.
\end{equation}
For the origins of the Burkholder functional we refer to \cite{Burk1,Burk2}. Its relation to Morrey's problem about rank-one and
quasiconvexity was first observed by T.~Iwaniec and it has since been the subject of intense investigation by a number of authors,
including \cite{AIPS12, AIPS15a, BarnsteinMontgomerySmith, Guerra19,GK,Sverak91,Tade}.
See also \cite{AIMbook,Banuelos, BJ,BJV,NV,PV02} for results about the norm of the Beurling-Ahlfors operator and the Iwaniec conjecture.
We emphasize the rather unusual feature of $\B_{K}^{\loc}$ that it is nonpositive and $p$-homogeneous on its effective domain $Q_{2}(K)$.

As is well known, a planar $K$-quasiregular map lies in $\WW^{1,p}_{\loc}$  for every $p < p_{K} \equiv 2K/(K-1)$,
while in general the integrability of the derivative of such a map fails at the borderline exponent $p_K$. Instead, in general its derivative
is locally in the Marcinkiewicz space $\LL^{p_{K},\infty}$ (see \cite{Astala94}).
On the other hand, an interesting result of \cite{AIPS12} states that for $K$-quasiregular maps $f \colon \Omega \to \C$ the function $\B_{p_K}( \D f)$
is locally integrable on $\Omega$, that is,
$$
0 \geq \int_{C} \! \B_{p_K}( \D f) \, \dd m(z) > -\infty
$$
holds for compact subsets $C$ of $\Omega$. However, even when $K$-quasiregular maps satisfy $f_j \rightharpoonup f$ weakly in $\WW^{1,p_K}$ it still is
not clear whether $\bigl( \B_{p_K}( \D f_j) \bigr)$ is locally equi-integrable. Thus the closed quasiconvexity of $\B_{K}^{\loc}$, proven
in  \cite[Theorem 1.4]{AFGKKpreprint}, led there to results that correspond to the $\WW^{1,q}$-swlsc of $\mathbf{B}_{K}^{\loc}$ only for $q> p_K$.
 
In this note we in particular extend, using principal quasiconvexity and Stoilow factorization, the $\WW^{1,q}$-swlsc of
the local Burkholder functional up to the borderline case $q = p_K$. This is a consequence of a stronger and in a sense more natural
result. To explain why our result is natural we fix a bounded domain $\Omega$ of $\C$ and a $K$-quasiregular map $g \colon \C \to \C$,
where $K>1$. Let $p=p_{K}$ be the corresponding borderline exponent and consider the variational problem of minimizing the Burkholder energy
$$
\mathcal{B}[v] \equiv \int_{\Omega} \! \B_{p}( \D v) \, \dd m(z)
$$
over all $K$-quasiregular maps $v \in g+\WW^{1,2}_{0}( \Omega )$. First one notes that from \cite{AIPS12} (see also Proposition \ref{prop:l1bound}
below for a more precise lower bound) it follows
that the Burkholder energy $\mathcal{B}[v]$ is bounded below over $v \in g+\WW^{1,2}_{0}( \Omega )$.
Next, let $(u_{j}) \subset  g+\WW^{1,2}_{0}( \Omega )$ be a minimizing sequence, so that in particular each $u_{j}\colon \Omega \to \C$ is $K$-quasiregular and,
in view of the Dirichlet condition, the sequence $( u_{j})$ is bounded in $\WW^{1,2}( \Omega )$.
By a normality result for $K$-quasiregular maps \cite{AIMbook}, $( u_{j})$ admits
a subsequence $( u_{j_k})$ that converges uniformly to a $K$-quasiregular map $u \in g+\WW^{1,2}_{0}( \Omega )$. In order to conclude that
$u$ is a minimizer one therefore needs lower semicontinuity of $\mathcal{B}[v]$ along the uniformly convergent subsequence $( u_{j_k})$.
Apart from the \textit{technical assumption} on the limit map $u$ the following result yields exactly that.

\begin{theorem}\label{swlscBurk}
Let $p > 2$, $K \equiv p/(p-2)$, $\Omega$ be a domain in $\C$ and suppose $u_j$, $u \colon \Omega \to \C$ are $K$-quasiregular maps
such that $u_{j} \to u$ locally uniformly on $\Omega$.
If $u \in \WW^{1,p}_{\loc}( \Omega )$, then
\begin{equation}\label{eq:swlscBurk}
  \int_{\Omega^\prime} \! \B_{p}( \D u ) \, \dd m(z) \leq \liminf_{j \to \infty} \int_{\Omega^\prime} \! \B_{p}( \D u_j ) \, \dd m(z)
\end{equation}
holds for each subset $\Omega^\prime \Subset \Omega$ with $\Leb^{2}(\partial \Omega^{\prime})=0$.
\end{theorem}
Since any locally uniform limit $u$ of a sequence $(u_j)$ of $K$-quasiregular maps is itself $K$-quasiregular, we see from  \cite{Astala94} that
the limit function $u$ in Theorem \ref{swlscBurk} automatically has the Marcinkiewicz-regularity $\D u \in \LL^{p,\infty}_{\loc}$. It remains open if  merely this   
suffices in Theorem \ref{swlscBurk}.

We refer to  Corollary \ref{betterBurk} for a global version of the Burkholder lower semicontinuity result on Dirichlet classes.
Since we do not know if the sequence $\bigl( \mathbf{B}_p(\D u_j) \bigr)$ in Theorem \ref{swlscBurk} is locally equi-integrable,
we need to handle the concentration and oscillation phenomena simultaneously. In this setting we will actually prove that the singular
part of the reduced defect measure for the sequence $\bigl( \B_{p}( \D u_{j}) \bigr)$ is zero.
\smallskip

We next consider lower semicontinuity results for general principal quasiconvex functionals, which are consistent with the blow up condition
\begin{equation}
\label{eq:blowup}
{ \bf E}(A) \textrm{ diverges as } \det A \to 0^+.
\end{equation}
For simplicity we restrict attention here to energy densities satisfying the following growth conditions: There exist $C>0$ and $p \in [1,2)$
such that
\begin{equation}\label{eq:growthconditions}
\left\{
\begin{array}{l}
  \bigl| {\bf E}(A)\bigr| \leq C \max \bigl\{|A|^{p}, -\log(\det A), K_A \bigr\}+C \, \mbox{ on } \, \M_{+},\\
{}\\
  {\bf E}(A) = +\infty \, \mbox{ on } \, \M \setminus \bigl( \M_{+} \cup \{ 0 \} \bigr) .
\end{array}
\right.
\end{equation}
Note we do not impose any condition for the value of ${\bf E}$ at $0$ and that the functional ${\bf W}$ defined at \eqref{wmagic} satisfies
\eqref{eq:growthconditions} with $p=1$.
With these bounds we have:

\begin{theorem}\label{thm:lscPQ}
Let ${\bf E} \colon \M \to \R \cup \{ +\infty \}$ be Borel measurable and satisfy \eqref{eq:growthconditions}. Assume that ${\bf E}$ is
$\WW^{1,1}$ principal quasiconvex. Fix an exponent $q>\frac{p}{2-p}$ and let $g \colon \C \to \C$ be a homeomorphism
of class $\WW^{1,2}_{\loc}( \C )$ with distortion $K_{g} \in \LL^{q}_{\loc}(\C )$.

Let $(u_j)$ be a sequence of homeomorphisms in $g+\WW^{1,2}_{0}(\Omega)$ with $\sup_j \|K_{u_j}\|_{\LL^q(\Omega)}<\infty$ and suppose that $u_j\weak u$
in $\WW^{1,2}( \Omega )$. Then $u \in g+\WW^{1,2}_{0}( \Omega )$ is a homeomorphism with $K_{u} \in \LL^{q}( \Omega )$ and
$$
\liminf_{j\to \infty} \int_\Omega {\bf E}(\D u_j(z)) \, \dd m(z) \geq \int_\Omega {\bf E} (\D u(z))\, \dd m(z).
$$
\end{theorem}
This result yields in a standard way an existence result for a variational problem related to hyperelasticity:

\begin{corollary}\label{existence}
Let ${\bf E} \colon \M \to \R \cup \{ +\infty \}$ be Borel measurable, satisfy \eqref{eq:growthconditions} and be $\WW^{1,1}$-principal quasiconvex.
Suppose ${\bf P} \colon \M \to \R \cup \{ +\infty \}$ is lower semicontinuous, polyconvex and satisfies the coercivity condition
\begin{equation}\label{coerciv}
{\bf P}(A) \geq c\bigl( |A|^{2}+K_{A}^{q} \bigr)
\end{equation}
for all $A \in \M$, where $c>0$, $q > \frac{p}{2-p}$ are constants. For a homeomorphism $g \colon \C \to \C$ of class $\WW^{1,2}_{\loc}( \C )$
and with ${\bf E}( \D g) + {\bf P}( \D g) \in \LL^{1}( \Omega )$ we put
$$
\mathcal{A}_{g} \equiv \biggl\{ u \in g+\WW^{1,2}_{0}( \Omega ) : \, u \colon \Omega \to g( \Omega ) \,
\mbox{ homeomorphism } \biggr\} .
$$
Then, for each $F \in \WW^{-1,2}( \Omega ) \equiv \WW^{1,2}_{0}( \Omega )^{\ast}$, the variational problem
$$
\inf_{u \in \mathcal{A}_{g}} \left( \int_{\Omega} \! \bigl( {\bf E}( \D u) + {\bf P}( \D u) \bigr) \, \dd m(z) + \bigl\langle F,u \bigr\rangle \right)
$$
admits a minimizer $u \in \mathcal{A}_{g}$ with $K_u \in \LL^{q}( \Omega )$.
\end{corollary}
Our last result provides lower semicontinuity for principal quasiconvex functionals along sequences of maps that
have asymptotically finite distortion in the following sense.

\begin{definition}
Let $( u_{j} )$ be a sequence in $\WW^{1,2}_{\loc}( \Omega )$ and $\A$ be a measurable subset of $\Omega$. We say   
that the maps have  \textit{asymptotically finite distortion} on $\A$ if
\begin{equation}\label{asfindist1}
\limsup_{j \to \infty} \mathbf{K}( \D u_{j}(x)) < +\infty \quad \mbox{ for a.e.} \; x \in \A.
\end{equation}
\end{definition}
We emphasize that the mappings do not necessarily have finite distortion everywhere, but that the distortion condition is merely required
to hold on the measurable set $\A$.

\begin{theorem}\label{asfinite-lsc1}
Let $\mathbf{E} \colon  \R^{2 \times 2} \to \R \cup \{ +\infty \}$ be lower semicontinuous and $\WW^{1,2}$-principal quasiconvex,
with ${\bf E} = +\infty$ on $\M \setminus \bigl( \M_{+} \cup \{ 0 \} \bigr)$, 
and assume that  
\begin{equation}\label{mysterious1}
\bigl| \mathbf{E}(A) \bigr| \leq \mathbf{C}(K_{A})\bigl( |A|^{2}+1 \bigr), \quad \forall A \in \M_{+} ,
\end{equation}
where $\mathbf{C} \colon [1,\infty ) \to [1,\infty )$ is an increasing function.
Let $\A$ be a measurable subset of $\Omega$.

Then, if $\,( u_{j})\, \subset \WW^{1,2}_{\loc}( \Omega )$ has asymptotically finite distortion on $\A$, if the sequence converges weakly to $u$
in $\WW^{1,2}_{\loc}( \Omega )$ and if 
\begin{equation}\label{equiint-energy1}
\bigl( \mathbf{E}(\D u_{j})^{-} \bigr) \, \mbox{ is equi-integrable on } \A ,
\end{equation}
then we have
$$
\liminf_{j \to \infty} \int_{\A} \! \mathbf{E}(  \D u_{j}) \, \dd m(z) \geq \int_{\A} \! \mathbf{E}( \D u) \, \dd m(z).
$$
\end{theorem}

\begin{remark}
Notice that we have not imposed any condition on the value of ${\bf E}$ at $0$, so any value for ${\bf E}(0)$
that is compatible with lower semicontinuity of ${\bf E}$ is permitted in Theorem \ref{asfinite-lsc1}.
\end{remark}

In this  introduction  we have stated the homogeneous versions of the lower semicontinuity estimates, but in the bulk of the text we will deal
with general non-autonomous $z$-dependent versions of the integrands. Although we have stated the results for planar maps, the reader
interested in applications could think of higher dimensional plane strains, which often can be modeled by two-dimensional deformations
\cite{BallJames91, OgdenBook, ORGDB20}. For example this could be relevant for isochoric Mooney-Rivlin materials altered in a
quasiconvex but non-polyconvex way; see \cite{BGO15, SN03} and the references therein.
We remark that our results in particular yield semicontinuity results for strain energies that decompose as a sum of a polyconvex energy density and
a principally quasiconvex energy density.
This is particularly relevant for a class of rank-one convex functionals with an additive isochoric-volumetric split, see \cite{Voss0,Voss1}
and \cite[Sect.~11]{AFGKKpreprint}.
We also remark that under natural coercivity conditions in terms of either the distortion or the second invariant one obtains existence of minimizers
subject to, for instance, Dirichlet boundary conditions. See \cite[Sect.~12]{AFGKKpreprint}.

{Before we embark on the proofs, we discuss and present examples of principal quasiconvex functionals in section~\ref{sec:examples}.
We start by recalling our earlier work on the Burkholder functional and isotropic functionals that admit a volumetric isochoric split.
Additionally, we show that for any functional depending only on distortion, principal quasiconvexity is equivalent to rank-one convexity and to polyconvexity.
We then investigate the natural question of which additional conditions, on top of quasiconvexity, imply principal quasiconvexity.
Here, Proposition~\ref{superharmonic} is the main technical result. The proof originates from \cite{AFGKKpreprint}, utilizing the theory of quadrature domains,
but in this work we deal more efficiently with possible concentration effects. This part combines the approximation result  \cite{HP18,Campbell17} together with some
careful manipulations with Orlicz functions in order to  transfer equiintegrability from $K_f$ to $\log(1+\frac{1}{J_f})$, in analogy with the $\LL^p$
estimates of Koskela and Onninen. More precisely, if ${\bf E}$ is $\WW^{1,\infty}$-quasiconvex, satisfies a natural $\LL^1$ type growth condition, see \eqref{supgrow}
in Section 2, and if for each $A \in \M_{+}$ the function
\begin{equation}\label{concavesintro}
\mathbb{C} \setminus \{ 0 \} \ni w \mapsto \mathbf{E}(wA) 
\end{equation}
is superharmonic, then ${\bf E}$ is also $\WW^{1,1}$-principal quasiconvex. 
For isotropic functionals the condition \eqref{concavesintro} simplifies further to an ordinary differential inequality, see Lemma~\ref{sh-iso}.
We end the examples section with a brief discussion of the so-called complex Burkholder functionals introduced in \cite{AIPS15b}. It is interesting to note
that these functionals have a different symmetry group, but still the arguments persist in this wider context.

In section~\ref{sec:Proofs} we describe the proofs for the semicontinuity and existence results. The proofs of
Theorems~\ref{thm:lscPQ} and \ref{asfinite-lsc1} follow a well-known strategy. First we use Young measures for the localization, we then deal with the
homogeneous gradient Young measures via Stoilow factorization, and finally we apply principal quasiconvexity. The localization for the case of asymptotic
finite distortion is however rather delicate, and involves a number of technical steps and results, where for instance the 0-1 law from \cite{AstalaFaraco02}
is used.  The proof of Theorem~\ref{swlscBurk} is more subtle as we cannot rely on gradient Young measure theory.
Instead we analyze the limit measure of the $\LL^1$-bounded sequence $\B_{p}( \D u_j)$ and localize differently at points where the limit measure is
absolutely continuous with respect to Lebesgue measure,  and at points where it is  singular. This blow up technique adjusts very well to principal quasiconvexity, though to treat
the singular part we need to assume critical integrability for the limit map (not for the sequence though).  Finally, we gather all results and prove in this
section the results of existence of minimizers. The reader not interested in proofs can directly check our local conditions and apply the existence theorem. 

\subsection*{Acknowledgments}
D.F, K.A, A.K acknowledge the financial support of QUAMAP, the ERC Advanced Grant 834728, 
 and of  
the Severo Ochoa Programme 
CEX2019-000904-S.
A.G. was supported by Dr.\ Max R\"ossler, the Walter Haefner Foundation and the ETH Z\"urich Foundation.  
D.F and A.K were  partially supported by CM and UAM,
and A.K by Academy of Finland CoE Randomness and Structures, and Academy Fellowship Grant 355840. 
 D.F acknowledges financial support by PI2021-124-195NB-C32. 
  D.F also wants to acknowledge the hospitality of the Mathematical Institute of Oxford during the summer of 2023.

\section{Examples of principal quasiconvex functionals}\label{sec:examples}

In this section we discuss and give examples of functionals satisfying the principal quasiconvexity condition of Definition \ref{def:pqc}.
Some of the examples are better understood using complex notation for  matrices $A \in \M$. Recall that this amounts to identifying $A$ with
$(a_{+},a_{-}) \in \mathbb{C}^2$, where, in terms of the usual identification of $\omega = ( \omega_{1},\omega_{2}) \in \mathbb{R}^2$ with
$\omega = \omega_{1}+\mathrm{i}\omega_{2} \in \C$,
\begin{equation}\label{eq:complexnotation}
A \omega=a_+  \omega+a_ - \overline{\, \omega \,}.
\end{equation}
Accordingly we get with $A^{+} \equiv (A + \mathrm{cof }A)/2$, $A^{-} \equiv (A - \mathrm{cof }A)/2$ that $A^{+} = (a_{+},0)$, $A^{-}=(0,a_{-})$
and this is the reason for calling $(a_{+},a_{-})$ the conformal-anticonformal coordinates of $A$. In this connection, we also record
that $( \rho \mathrm{e}^{\mathrm{i}\theta},0) = \rho R_{\theta}$, where $R_\theta \in \mathrm{SO}(2)$ represents the (anti-clockwise) $\theta$-angle rotation about
the origin. Similarly, the complex conjugate mapping, $\omega \mapsto \overline{\omega}$, is identified with the reflection $\mathrm{diag }(1,-1) = (0,1)$. 

For later use, we also recall that for $A \in \M$ we have the identities: $|A| = |a_{+}| + |a_{-}|$, $J_{A} = |a_{+}|^{2}-|a_{-}|^{2}$. When $A \in \M_{+}$ we
have $K_{A} = (|a_{+}|+|a_{-}|)/(|a_{+}|-|a_{-}|) = (1+| \mu_{A}|)/(1-| \mu_{A}|)$, where the ratio $\mu_{A} \equiv a_{-}/a_{+}$ is called the complex dilation of $A$.

A functional ${\bf E} \colon \M_{+} \to \R \cup \{ +\infty \}$ is isotropic provided ${\bf E}(QAR) = {\bf E}(A)$ holds for all $A \in \M_{+}$
and $Q$, $R \in \mathrm{SO}(2)$. There are a number of equivalent ways to characterize isotropy and here we emphasize that both $A \mapsto J_A$
and $A \mapsto K_A$ are isotropic. These two functionals can be considered the building blocks for all isotropic functionals in the sense that one can
show ${\bf E}$ to be isotropic precisely when ${\bf E}(A) = E(K_{A},J_{A})$ holds for all $A \in \M_{+}$ for some extended real-valued function
$E \colon (0,\infty )^{2} \to \R \cup \{ +\infty \}$.
Below we shall be considering subclasses of the isotropic functionals that are naturally expressed in such terms.
Finally, it will be convenient on a few occasions to use the short-hand $zA$ for the matrix $\rho R_\theta A = (z,0)(a_{+},a_{-})= (za_{+}, za_{-})$, where
$z = \rho \mathrm{e}^{\mathrm{i}\theta}$. For later reference we also record the general multiplication rule for matrices $A=(a_{+},a_{-})$ and $B = (b_{+},b_{-})$
in complex notation
\begin{equation}\label{cmultiply}
AB = (a_{+},a_{-})(b_{+},b_{-}) = \bigl( a_{+}b_{+} + a_{-}\overline{b_{-}}, a_{+}b_{-}+a_{-}\overline{b_{+}} \bigr) .
\end{equation}
A first simple class of examples is furnished by the following result:

\begin{proposition}\label{polycon}
Assume $P \colon \M \times \R \to \R \cup \{ +\infty \}$ is an extended real-valued, lower semicontinuous and convex function for which the partial function
$t \mapsto P(A,t)$ is non-increasing on $[0,|A|^{2}]$ for each fixed $A \in \M$.  Then the functional $\mathbf{P}(A) \equiv P(A,J_{A})$, $A \in \M_{+}$, is
$\WW^{1,1}$-principal quasiconvex.
\end{proposition}

\begin{proof}
Let $f \colon \C \to \C$ be an orientation-preserving homeomorphism of class $\WW^{1,1}_{\loc}( \C )$, which is conformal
on $\C \setminus \overline{\DD}$ with the Laurent expansion
\begin{equation}\label{laurent}
f(z) = b_{0}z + \frac{b_{1}}{z}+\sum_{n=2}^{\infty} \frac{b_{n}}{z^n} \, \mbox{ for } \, |z| > 1.
\end{equation}
Note that a principal map in the sense of our Definition \ref{def:principal} in particular will satisfy these conditions.
According to the area formula (see \cite[Th. 2.10.1 \& Cor. 2.10.2]{AIMbook}) we have, in terms of the notation introduced after
Definition \ref{def:principal}, that
\begin{equation}\label{areaformula}
\fint_{\DD} \! J_{f} \, \dd m(z) \leq \det A_{f} - \sum_{n=2}^{\infty} n|b_{n}|^{2},
\end{equation}
and therefore in particular, $0 \leq \fint_{\DD} \! J_{f} \, \dd m(z) \leq \det A_{f}$. We employ this inequality in conjunction with Jensen's
inequality whereby
$$
\fint_{\DD} \! \mathbf{P}( \D f) \, \dd m(z) \geq P \left( A_{f}, \fint_{\DD} \! J_{f} \, \dd m(z) \right) .
$$
Here we have $0 \leq \fint_{\DD} \! J_{u} \, \dd m(z) \leq \det A_{f} \leq |A_{f}|^2$, where the last inequality is Hadamard's. We may now conclude
with the desired inequality if we use that $t \mapsto P(A_{f},t)$ is non-increasing on $[0,|A_{f}|^{2}]$.
\end{proof}
We record that the above argument also illustrates why $A \mapsto J_{A}$ is not $\WW^{1,\infty}$-principal quasiconvex. Indeed the area inequality \eqref{areaformula}
immediately implies that $\fint_{\DD} \! J_{f} \, \dd m(z) < \det A_{f}$ whenever $f$ is a $\WW^{1,\infty}_{\loc}$ principal map where at least
one of the coefficients $b_n \neq 0$ for $n \geq 2$ in its Laurent expansion \eqref{laurent}.
Prime examples of polyconvex functionals that are also $\WW^{1,1}$-principal quasiconvex are, minus the determinant, $-J_{A}$, and the distortion
functional $\mathbf{K}(A) \equiv K_{A}$. It is well-known from \cite{BM1} that $A \mapsto -J_{A}$ is $\WW^{1,p}$-quasiconvex if and only if $p \geq 2$, so the fact
that we in the definition of $\WW^{1,1}$-principal quasiconvexity only test with special $\WW^{1,1}$ maps is essential. See also \cite{MuSiSp99}
for further discussion of the $\WW^{1,p}$-quasiconvexity condition that involves the distributional Jacobian.

The second invariant $\mathbf{I_{2}}(A) = K_{A}+1/K_{A}$ is also simultaneously polyconvex and $\WW^{1,1}$-principal quasiconvex. This is a consequence of the
following result concerning so-called isochoric, or conformally invariant, free energy functionals, meaning those isotropic functionals that are also homogeneous of
degree $0$. The canonical example of an isochoric functional is the distortion $A \mapsto K_A$ and it can be shown that ${\bf H} \colon \M_{+} \to \R \cup \{ +\infty \}$
is isochoric precisely when ${\bf H}(A) = H(K_{A})$ for all $A \in \M_{+}$, where $H \colon [1,\infty ) \to \R \cup \{ +\infty \}$. It is sometimes convenient to
use the equivalent characterization in terms of the second invariant. Namely, ${\bf H}$ is isochoric precisely when ${\bf H}(A) = \tilde{H}( {\bf I}_{A})$,
where $\tilde{H} \colon [2,\infty ) \to \R \cup \{ +\infty \}$. For instance, see \cite{Voss0} and the references therein. 
This class of energy functionals and the associated variational problems have been discussed by a number of authors,
including \cite[Chap.~21]{AIMbook}, \cite{AIMO05}, \cite{MartinYao24} and \cite{Voss0,Voss1,Voss2}. The following result is stated in terms of the distortion
function, but an entirely analoguous statement can be given in terms of the second invariant. We leave the details of this to the interested reader.

\begin{proposition}\label{qc-isochoric}
Let $H \colon [1,\infty ) \to \R \cup \{ +\infty \}$ be a lower semicontinuous and extended real-valued function. Define ${\bf H}(A) \equiv H(K_{A})$, $A \in \M_{+}$.
Then the following are equivalent:    
\begin{itemize}
\item[(i)] ${\bf H}$ is rank-one convex;
\item[(ii)] ${\bf H}$ is $\WW^{1,1}$-principal quasiconvex;
\item[(iii)] ${\bf H}$ is polyconvex;
\item[(iv)] $H$ is convex and non-decreasing.
\end{itemize}
\end{proposition}

\begin{proof}
The pairwise equivalences of (i), (iii) and (iv) follow from \cite{Voss0}. To see the equivalence to (ii) we first observe that an argument similar to the
one used in the proof of Proposition \ref{polycon} yields (ii) from (iv). We therefore conclude the proof if we can show that (i) follows from (ii).
Assume that (ii) holds and fix matrices $A_0$, $A_{1} \in \M_{+}$ that are rank-one connected. Put for a weight $\lambda \in (0,1)$,
$A_{\lambda} = (1-\lambda )A_{0}+\lambda A_{1}$. We can without loss in generality assume that ${\bf H}(A_{0})$, ${\bf H}(A_{1}) < +\infty$. Note that $A_{\lambda} \in \M_{+}$
and $1 \leq K_{A_i} < +\infty$, hence by rank-one convexity of the distortion, $K_{A_{\lambda}} \leq (1-\lambda )K_{A_0}+\lambda K_{A_1} \leq K \equiv \max \{ K_{A_0},K_{A_1} \} < +\infty$.
Write $A_{1}-A_{0} = a \otimes n$ with $|n|=1$, take $\theta \in \mathrm{SO}(2)$ with $\theta e_{1} = n$, denote $Q = (0,1)^2$ and $R = \theta \bigl( Q \bigr)$.
Define, in terms of real notation,
$$
v(z) \equiv \left\{
\begin{array}{ll}
  ( A_{0}-A_{\lambda}) z + (1-\lambda )a & \mbox{ in } R_{0},\\
  ( A_{1}-A_{\lambda}) z & \mbox{ in } R_{1},
\end{array}
\right.
$$
where $R_{i} = \theta \bigl( Q_{i} \bigr)$, $Q_{0} = (0,1-\lambda ) \times (0,1)$, $Q_{1} = ( 1-\lambda ,1) \times (0,1)$.
Observe that hereby $v \colon R \to \R^2$ is a Lipschitz map that is constant in the direction perpendicular to $n$ and
$v\bigl( \theta (te_{2})\bigr) = v\bigl( n+\theta (te_{2})\bigr)$ for $t \in (0,1)$. If therefore we extend $v$ to $\R^2$ by $R$-periodicity, then we obtain
a Lipschitz and $R$-periodic map that we again denote $v \colon \R^2 \to \R^2$. By inspection
$$
\D v = \left\{
\begin{array}{ll}
  A_{0}-A_{\lambda} & \mbox{ in } R_{0},\\
  A_{1}-A_{\lambda} & \mbox{ in } R_{1},\\
\end{array}
\right.
$$
hence $\int_{R} \! \D v \, \dd m(z)=0$. Put $u_{j}(z) \equiv A_{\lambda}z + \frac{1}{j}v(jz)$, $z \in \DD$, and record that $u_j$ is Lipschitz with
$$
\D u_{j} = \left\{
\begin{array}{ll}
  A_{0} & \mbox{ in } R_{0,j},\\
  A_{1} & \mbox{ in } R_{1,j},\\
\end{array}
\right.
$$
where the disjoint sets $R_{0,j}$, $R_{1,j} \subset \DD$ satisfy $\mathscr{L}^{2}(R_{0,j}) = (1-\lambda )\pi + \mathrm{o}(1)$,
$\mathscr{L}^{2}(R_{1,j}) = \lambda \pi + \mathrm{o}(1)$. In particular, $K_{u_j} \leq K$, so $u_j$ is $K$-quasiregular and
we may Stoilow factorize it: $u_{j} = h_{j} \circ f_{j}$ on $\DD$, where $f_{j} \colon \C \to \C$ is a $K$-quasiconformal
normalized principal map with $K_{f_j}=K_{u_j}$ and $h_{j} \colon f_{j}( \DD ) \to \C$ is holomorphic. It follows
from the Riemann-Lebesgue lemma that $u_{j} \weakstar A_{\lambda}$ in $\WW^{1,\infty}( \DD )$ and so by uniqueness of the Stoilow
factorization we find that $f_{j} \to f$ uniformly on $\C$ and weakly in $\WW^{1,2}_{\loc}( \C )$ and $h_{j} \to h$ locally uniformly
on $f( \DD )$. Here we have in terms of $A_{\lambda} = (a_{\lambda ,+},a_{\lambda ,-})$ and in complex notation,
$$
f(z) = \left\{
\begin{array}{ll}
  z+\frac{a_{\lambda ,-}}{a_{\lambda ,+}}\bar{z} & \mbox{ if } |z| \leq 1,\\
  z+\frac{a_{\lambda ,-}}{a_{\lambda ,+}}\frac{1}{z} & \mbox{ if } |z| > 1,\\
\end{array}
\right.
\quad \mbox{ and } \quad h(z) = a_{\lambda ,+}z.
$$
Because
$$
K_{f_j} = K_{u_j} = \left\{
\begin{array}{l}
  K_{A_0} \mbox{ on } R_{0,j},\\
  K_{A_1} \mbox{ on } R_{1,j},
\end{array}
\right.
$$
we get
$$
\fint_{\DD} \! {\bf H}( \D f_{j}) \, \dd m(z) = (1-\lambda ){\bf H}(A_{0})+\lambda {\bf H}(A_{1}) + \mathrm{o}(1)
$$
and here the left-hand side is estimated using $\WW^{1,1}$-principal quasiconvexity,
$$
{\bf H}\left( \fint_{\DD} \! \D f_{j} \, \dd m(z) \right)  \leq \fint_{\DD} \! {\bf H}( \D f_{j}) \, \dd m(z).
$$
Because $\fint_{\DD} \! \D f_{j} \, \dd m(z) \to A_{f} = A_{\lambda}/a_{\lambda ,+}$ and since $K_{A_{\lambda}/a_{\lambda ,+}} = K_{A_{\lambda}}$ we
may use the lower semicontinuity of $H$ to conclude.
\end{proof}
Next we turn to the class of free energy functionals that are isotropic and admit an additive isochoric-volumetric split, meaning they must
have the general form $${\bf E}(A) = H(K_{A})+G(J_{A}), \quad A \in \M_{+},$$ where $H$ and $G$ are extended real-valued functions. Such free energy
functionals have been considered by many authors and are often used to model slightly compressible materials, see for instance
\cite{Flory,ChDaHaLa88,HartmannNeff,OgdenBook}.
Whereas the case of rank-one convexity for the individual terms $H(K_{A})$ and $G(J_{A})$ is characterized by Propositions \ref{qc-isochoric} and \ref{polycon}, respectively,
the rank-one convexity of ${\bf E}(A)$ is more involved and does not necessarily mean that both $H(K_{A})$ and $G(J_{A})$ are rank-one convex.
The precise conditions have been worked out in the interesting paper \cite[Th.~2.6]{Voss0} and following \cite{Voss1} we shall here focus on the special
case when the isochoric part is rank-one convex. In that situation it was shown in \cite{Voss1} (see also \cite[Sect.~11]{AFGKKpreprint})
that ${\bf E}$ is rank-one convex if and only if there exists an isotropic polyconvex functional ${\bf P}$ and a non-negative constant $c \geq 0$ such that
${\bf E}={\bf P} + c{\bf W}$, where
\begin{equation}\label{wmagic}
  {\bf W}(A) \equiv K_{A} -\log K_{A} + \log J_{A} \, \mbox{ for } A \in \M_{+},
\end{equation}
as observed in \cite{Voss1}, is an isotropic rank-one convex, non-polyconvex functional. Indeed it cannot be polyconvex because ${\bf W}(tA) \to -\infty$ as $t \to 0$.
If we follow the convention \eqref{convention} of this paper and extend ${\bf W}$ to full matrix space $\M$ and at the same time insist that the extension be
lower semicontinuous, then we must define
$$
{\bf W}(A) = \left\{
\begin{array}{ll}
  K_{A} -\log K_{A} + \log J_{A} & \mbox{ if } A \in \M_{+},\\
  -\infty & \mbox{ if } A=0,\\
  +\infty & \mbox{ if } A \in \M \setminus \bigl( \M_{+} \cup \{ 0 \} \bigr) .
\end{array}
\right.
$$
We are therefore forced to consider extended real-valued integrands that assume both $\pm \infty$. This example, together with the local Burkholder
functional \eqref{localburk}, are our reasons for adopting an ad hoc approach to the value at $0$ when extending functionals from $\M_{+}$ to $\M$.
We remark that this extension of ${\bf W}$ remains isotropic and rank-one convex (with the natural definitions, see for instance \cite{AFGKKpreprint}).
The following result is proved in \cite[Cor.~1.7, Th.~1.8 \& Th.~1.9]{AFGKKpreprint}. We emphasize that it in particular yields a positive solution to
Morrey's problem within the considered class of functionals.

\begin{theorem}\label{isochoric-volumetric}
Let ${\bf E}(A) = H(K_{A})+G(J_{A})$, $A \in \M_{+}$, where $H \colon [1,\infty ) \to \R \cup \{ +\infty \}$ and $G \colon (0, \infty ) \to \R \cup \{ +\infty \}$
are lower semicontinuous extended real-valued functions. Assume $H$ is convex.
Then ${\bf E}$ is rank-one convex if and only if ${\bf E}$ is $\WW^{1,2}$-quasiconvex. 

Furthermore, the functional ${\bf W}$ from \eqref{wmagic}  is $\WW^{1,1}$-principal quasiconvex.
\end{theorem}
We return to the Burkholder functional $\B_p \colon \M \to \R$ that was defined for exponents $p \geq 2$ at \eqref{burkholder} in the introduction.
In complex notation it is given by
\begin{equation}\label{burkholder2}
\mathbf{B}_p(A)= \bigl( (p-1)|a_-|-|a_+| \bigr) \bigl( |a_+|+|a_-| \bigr)^{p-1}
\end{equation}
for a matrix $A \in \M$ that has conformal-anticonformal coordinates $(a_+,a_-) \in \C^2$. It is isotropic, $p$-homogeneous and rank-one convex, but
it is an open question whether it is also $\WW^{1,\infty}$-quasiconvex. Its quasiconvexity would among other things provide an alternative approach
to the optimal higher integrability properties of quasiconformal maps \cite{Astala94}, see \cite{Tade} where the reader can also find interesting extensions
to higher dimensions $n > 2$. In the recent work, \cite{AFGKKpreprint}, building on \cite{AIPS12,AIPS15a}, it was shown that the Burkholder functional is
quasiconvex when restricted to quasiconformal test functions with appropriate distortion. The precise result is conveniently stated in terms of the
local Burkholder functional.

\begin{theorem}\label{locburkisqc}
Let $p>2$, $p^{\prime} = p/(p-1)$ and $K = p/(p-2)$. The local Burkholder functional $\B_{K}^{\loc}$ defined at \eqref{localburk} is simultaneously
$\WW^{1,1}$-principal quasiconvex, $\WW^{1,p^{\prime}}$-quasiconvex and closed $\WW^{1,q}$-quasiconvex for each $q > p^{\prime}$.
\end{theorem}
We refer to \cite[Th.~1.4 \& Th.~1.6]{AFGKKpreprint} for the proofs and note that the statement is essentially sharp.
Indeed, if we test $\B_{K}^{\loc}$ with the staircase laminate introduced in \cite[Example 5.4]{DF1} we see that it is not closed $\WW^{1,q}$-quasiconvex for $q < 2K/(K+1)$.
Moreover, if we test $\B_{K}^{\loc}$ with the map obtained in \cite[Theorem 3.2]{AFS08} we see that it is not even $\WW^{1,q}$-quasiconvex for
$q < 2K/(K+1)$. That $\B_{K}^{\loc}$ can be $\WW^{1,1}$-principal quasiconvex is because we only require Jensen's inequality for homeomorphisms in the definition
of principal quasiconvexity and a quasiconformal homeomorphism of Sobolev class $\WW^{1,1}_{\loc}$ is automatically in $\WW^{1,2}_{\loc}$.\footnote{This is classical.
A proof follows by applying, for example \cite[Corollary 3.3.6]{AIMbook}, to get that the Jacobian is in $\LL^1$ and deduce the square integrability of the
full derivative by the distortion inequality.}

We have already noted that $\B_{K}^{\loc}$ is non-positive on its effective domain which is
the $K$-quasiconformal well $Q_{2}(K)$. Indeed, it is $0$ on its boundary and strictly negative in its interior. If we combine Theorem \ref{locburkisqc} with
Jensen's inequality for convex functions we deduce the following corollary.


\begin{corollary}\label{clearly}
Let $p>2$, $p^{\prime} = p/(p-1)$ and $K = p/(p-2)$. Assume $\theta \colon ( -\infty , 0) \to \R \cup \{ +\infty \}$ is a lower semicontinuous convex
and non-decreasing function satisfying $\theta (t) \to +\infty$ as $t \to 0-$. Then the functional $\B_{p}^{\theta} \colon \M \to \R \cup \{ +\infty \}$
defined by
$$
\B^{\theta}_{p}(A) \equiv \left\{
\begin{array}{ll}
\theta \bigl( \B_{p}(A) \bigr) & \mbox{ if } A \in \M_{+} \mbox{ and } K_{A} < K,\\
+\infty & \mbox{ if } A \in \M \mbox{ and } A=0 \mbox{ or } K_{A} \geq K,
\end{array}
\right.
$$
is $\WW^{1,1}$-principal quasiconvex, $\WW^{1,p^{\prime}}$-quasiconvex and closed $\WW^{1,q}$-quasiconvex for each $q>p^{\prime}$. 
\end{corollary}
We leave the elementary verification to the interested reader. We shall exhibit one particular example from this and start by
recalling that $\B_p$ is isotropic. If therefore we restrict $\B_{p}(A)$ to $A \in \M_{+}$ it can be expressed in terms of the
distortion and the Jacobian of $A$, where, if we assume $p>2$ and keep writing $K = p/(p-2)$, it becomes
$$
\B_{p}(A) = \frac{p-2}{2} \bigl( K_{A}-K \bigr) K_{A}^{\frac{p-2}{2}}J_{A}^{\frac{p}{2}}, \quad A \in \M_{+} .
$$
It clearly does not admit an additive isochoric-volumetric split. However, it does admit a multiplicative isochoric-volumetric split
as
$$
\B_{p}(A) =  \Phi (K_{A}) \Psi (J_{A}),
$$
where $\Phi (s) = \frac{p-2}{2}(s-K)s^{\frac{p-2}{2}}$, $s \geq 1$, and $\Psi (t)=t^{\frac{p}{2}}$, $t>0$. It is interesting to note that
both $\Phi$ and $\Psi$ are increasing and convex on their respective domains. We exploit this as follows and note first that
$\Phi (s) <0$ for $s \in [1,K)$ corresponding to $\B_{p}(A) < 0$ on the interior of the $K$-quasiconformal well. We may therefore define
the functional $\B^{\theta}_{p}$ corresponding to $\theta (t) = -\log (-t)$, $t<0$, whereby the functional
\begin{equation}\label{logburk}
\mathbf{L}_{p}(A) \equiv \B_{p}^{\theta}(A) = \left\{
\begin{array}{ll}
 -\log \bigl( - \Phi (K_{A}) \bigr) - \log \bigl( \Psi (J_{A}) \bigr)  & \mbox{ if } A \in \M_{+} \mbox{ and } K_{A} < K,\\
  +\infty & \mbox{ otherwise,}
\end{array}
\right.
\end{equation}
results and it clearly has the additive isochoric-volumetric split form. By virtue of Propositions \ref{polycon} and \ref{qc-isochoric}
both the isochoric and volumetric parts are polyconvex, so $\mathbf{L}_{p}$ is in particular polyconvex too. 

The following proposition gives, modulo a one-sided growth condition, a local condition for a quasiconvex function to be automatically principal
quasiconvex. The interested reader will see that the statement in particular applies to the functional ${\bf W}$ defined at \eqref{wmagic}.

\begin{proposition}\label{superharmonic}
Let ${\bf E} \colon \M_{+} \to \R$ be $\WW^{1,\infty}$-quasiconvex. Suppose that
\begin{equation}\label{supgrow}
{\bf E}(A) \leq c\left( 1+ K_{A}+ \log \left( 1+ \frac{1}{J_{A}}\right) +J_{A}^{\alpha} \right)
\end{equation}
and that the function
\begin{equation}\label{shh}
\C \setminus \{ 0 \} \ni w \mapsto {\bf E}(wA) \, \mbox{ is superharmonic}
\end{equation}
for all $A \in \M_{+}$, where $c>1$ and $\alpha \in (0,1)$ are constants.
Then ${\bf E}$ is $\WW^{1,1}$-principal quasiconvex.
\end{proposition}
One may give versions of this result also for functionals defined on quasiconformal wells and that are modelled on the local
Burkholder functionals; we intend to report on that elsewhere. It is possible to simplify the proof of Proposition \ref{superharmonic}
below if we replace the condition \eqref{shh} by the stronger condition that the function
\begin{equation}\label{concaves}
  \C \setminus \{ 0 \} \ni w \mapsto {\bf E}(wA) \, \mbox{ is locally concave}.
\end{equation}
Local concavity amounts to concavity of the function near each $w_{0} \neq 0$ and is used here because the set $\C \setminus \{ 0 \}$ clearly
is not convex. It suffices for \eqref{sh} below since it allows us, via Jensen's inequality, to check that the composite function
$z \mapsto {\bf E}(h^{\prime}(z)A)$ locally satisfies the mean-value inequality that is known to be equivalent to superharmonicity.
However, as ${\bf E}$ is real-valued and quasiconvex it is also rank-one convex and so locally Lipschitz. It is then not difficult to see
that the local concavity \eqref{concaves} implies concavity in the sense that ${\bf E}(wA)$ lies below its supporting hyperplanes 
everywhere on $\C \setminus \{ 0 \}$. The interested reader will see that this observation allows a simplification of the proof below.

The following lemma characterizes the condition \eqref{shh} in the case the functional ${\bf E}$ is isotropic.
Its proof is a straightforward calculation that is left to the interested reader.
\begin{lemma}\label{sh-iso}
Suppose ${\bf E}(A) = E(K_{A},J_{A})$,  $A \in \M_{+}$, where $E \colon [1,\infty ) \times (0,\infty ) \to \R$ is locally Lipschitz.
Then \eqref{shh} holds if and only if the function $(s,t) \mapsto E(s,t)$ satisfies
$$
t\partial_{t}^{2}E + \partial_{t}E \leq 0
$$
in the sense of distributions.
\end{lemma}

The proof of Proposition~\ref{superharmonic} requires a number of further results that we turn to next and in particular
it also uses elements from \cite{AFGKKpreprint}. The most important is \cite[Lemma 7.2]{AFGKKpreprint} that we for convenience of the
reader state here in a form adapted to our setting.

\begin{lemma}\label{karitrick}
Suppose that $f \colon \C \to \C$ is a principal map of class $\WW^{1,1}_{\loc}( \C )$ with Laurent expansion \eqref{laurent}
in the complement of the unit disk. Put $R_{f}(z) \equiv b_{0}z + b_{1}/z$ and $h \equiv f \circ R_{f}^{-1}$ on $\C \setminus A_{f}( \overline{\DD})$.
Then $h \colon \C \setminus A_{f}( \overline{\DD}) \to f( \C \setminus \overline{\DD})$ is a conformal homeomorphism and
$$
g \equiv \left\{
\begin{array}{ll}
  f & \mbox{ in } \overline{\DD},\\
  h \circ A_{f} & \mbox{ in } \C \setminus \overline{\DD},
\end{array}
\right.
$$
is a homeomorphism of class $\WW^{1,1}_{\loc}( \C )$. Moreover, $h(z) = z + \mathrm{O}(z^{-2})$ as $|z| \to \infty$.
\end{lemma}

The next auxiliary lemma concerns an elementary implication of the superharmonicity condition \eqref{shh}.

\begin{lemma}\label{shfromshh}
Assume ${\bf E} \colon \M_{+} \to \R$ is locally Lipschitz and satisfies \eqref{shh}. Then for each domain $\Omega$ in $\C$
and conformal map $h \colon \Omega \to \C$ the composite map
\begin{equation}\label{sh}
  \Omega \ni z \mapsto {\bf E}(h^{\prime}(z)A)
\end{equation}  
is superharmonic for each $A \in \M_{+}$.
\end{lemma}

\begin{proof}
Using mollification if necessary we may assume without loss of generality that ${\bf E}$ is smooth. If we calculate the
Laplacian in polar coordinates we see that \eqref{shh} is equivalent to validity of
$$
0 \geq {\bf E}^{\prime \prime}(wA)[wA,wA] + {\bf E}^{\prime \prime}(wA)[wJA,wJA]
$$
for all $w \in \C \setminus \{ 0 \}$ and $A \in \M_{+}$. Here $J$ is the anti-clockwise rotation of $\pi /2$ about the origin.
Calculating the Laplacian in cartesian coordinates we find for $h$ as in the statement that
$$
\Delta {\bf E}(h^{\prime}(z)A) = {\bf E}^{\prime \prime}(h^{\prime}(z)A)[ \partial_{x}\nabla h A,\partial_{x}\nabla hA] +
{\bf E}^{\prime \prime}(h^{\prime}(z)A)[ \partial_{y}\nabla h A,\partial_{y}\nabla hA]
$$
holds. Using the Cauchy-Riemann equation we check that $\partial_{y}\nabla h = \partial_{x}\nabla h J$, and this concludes the proof.
\end{proof}

The following result might be well-known to experts, but as we could not find it stated anywhere in the literature we include it here
for completeness.

\begin{proposition}\label{upgradeapprox}
Let $\Omega$, $\Lambda$ be bounded domains in $\C$ and $f \colon \overline{\Omega} \to \overline{\Lambda}$ be a homeomorphism
such that its restriction to $\Omega$ is orientation-preserving and of class $\WW^{1,2}( \Omega )$. Assume that $( f_{n})$ is
a sequence of homeomorphisms of $\overline{\Omega}$ onto $\overline{\Lambda}$ such that $f_{n} = f$ on $\partial \Omega$,
the restriction of $f_n$ to $\Omega$ is a $\CC^1$-diffeomorphism of $\Omega$ onto $\Lambda$, and $f_{n} \to f$ uniformly
on $\overline{\Omega}$ and strongly in $\WW^{1,2}( \Omega )$.

Then each $f_{n}^{-1}$, $f^{-1} \in \WW^{1,1}( \Lambda )$, $f_{n}^{-1} \to f^{-1}$ uniformly on $\overline{\Lambda}$ and strongly
in $\WW^{1,1}( \Lambda )$.
\end{proposition}

\begin{proof}
We start by showing the uniform convergence and fix a sequence $w_{n} \to w$ in $\overline{\Lambda}$. Since $f_{n}^{-1}(w_{n}) \in \overline{\Omega}$
it follows by compactness that any given subsequence admits a further subsequence that converges. Thus through a suitable subsequence we have
that $f_{n}^{-1}(w_{n}) \to z$ in $\overline{\Omega}$. By uniform convergence of $f_{n}$ to $f$ we then get $w=f(z)$, and thus $z=f^{-1}(w)$.
A routine subsequence argument now allows us to conclude that $f_{n}^{-1} \to f^{-1}$ uniformly on $\overline{\Lambda}$.

We turn to the strong convergence in $\WW^{1,1}( \Lambda )$ and note first that according to \cite{HK} we have $f^{-1} \in \WW^{1,1}( \Lambda )$,
$J_{f^{-1}} > 0$ a.e.~and
$$
\int_{\Omega} \! \| \D f\| \, \dd m(z) = \int_{\Lambda} \! \| \D (f^{-1}) \| \, \dd m(w).
$$
Here $\| \cdot \|$ denotes the Frobenius norm on $\M$. Because especially $\D f_{n} \to \D f$ strongly in $\LL^{1}( \Omega )$ we get,
using also that $f_n$ are $\CC^1$-diffeomorphisms and a change-of-variables, that
\begin{equation}\label{extractionpoint1}
\int_{\Lambda} \! \| \D (f_{n}^{-1}) \| \, \dd m(w) \to \int_{\Lambda} \! \| \D (f^{-1}) \| \, \dd m(w).
\end{equation}
It follows from this that the sequence $\bigl( \D (f_{n}^{-1}) \bigr)$ is equi-integrable on $\Lambda$. We shall prove this
and the strong convergence using generalized Young measures and for this we rely on notation and results from \cite{AlBo}.
Our notation and the set-up for generalized Young measures follow precisely that of \cite{KrRa22}, which is adapted from \cite{AlBo}.
Because $\bigl( \D (f_{n}^{-1}) \bigr)$ is bounded in $\LL^{1}( \Lambda , \M )$ it follows that any given subsequence admits
a further subsequence, for convenience not relabelled, such that for a generalized Young measure $\nu = ( \nu_{w},\lambda ,\nu_{w}^{\infty})$ we
have
$$
\D ( f_{n}^{-1}) \stackrel{{\rm Y}}{\to} \bigl( \nu_{w}, \lambda , \nu_{w}^{\infty} \bigr) ,
$$
see \cite[Sect.~2.4 and Lemma 2.3]{KrRa22}. We have $\D (f_{n}^{-1}) \to \D (f^{-1})$ in the sense of distributions on $\Lambda$ and in view
of the $\LL^1$ bound the convergence is weak$\mbox{}^\ast$ in $\CC_{0}( \Lambda )^\ast$. We may then identify the weak$\mbox{}^\ast$ limit using
the Young measure as follows:
$$
\langle \nu_{w},\mathrm{Id} \rangle = \D (f^{-1})(w) \quad \mbox{$\mathscr{L}^{2}$-a.e.~in $\Lambda$, and }
\langle \nu_{w}^{\infty}, \mathrm{Id} \rangle = 0 \quad \mbox{$\lambda$-a.e.~in $\Lambda$}
$$
From \eqref{extractionpoint1} we get for the Young measure the identity
$$
\int_{\Lambda} \! \langle \nu_{w}, \| \cdot \| \rangle \, \dd m(w) + \lambda ( \overline{\Lambda}) =
\int_{\Lambda} \! \| \langle \nu_{w},\mathrm{Id} \rangle \| \, \dd m(w),
$$
and so, using the convexity of the Frobenious norm, $\lambda (\overline{\Lambda})=0$ and
$\langle \nu_{w}, \| \cdot \| \rangle = \| \langle \nu_{w},\mathrm{Id} \rangle \|$ for $\mathscr{L}^2$-a.e.~$w \in \Lambda$.
The former means that the (sub-)sequence $\bigl( \D (f_{n}^{-1}) \bigr)$ is equi-integrable on $\Lambda$, whereas the latter
only gives that $\nu_{w}$ is supported in the half-line $\{ t\langle \nu_{w},\mathrm{Id} \rangle : \, t \geq 0 \}$ for
$\mathscr{L}^2$-a.e.~$w \in \Lambda$ (see \cite{AlBo}). We assert that $\nu_{w}=\delta_{\D (f^{-1})(w)}$ for $\mathscr{L}^2$-a.e.~$w \in \Lambda$.
It is routine that this will conclude the proof. Now
$$
\D (f_{n}^{-1})(w) = \D f_{n}(f_{n}^{-1}(w))^{-1}
$$
and so for $\phi \in \CC ( \overline{\Lambda})$, $\Phi \in \CC_{0}( \M )$ we calculate using the change-of-variables $z=f_{n}^{-1}(w)$:
$$
\int_{\Lambda} \! \phi (w) \Phi ( \D (f_{n}^{-1})(w)) \, \dd m(w) = \int_{\Omega} \! \phi( f_{n}(z)) \Phi ( \D f_{n}(z)^{-1}) J_{f_{n}}(z) \, \dd m(z).
$$
Since $f_{n} \to f$ strongly in $\WW^{1,2}( \Omega )$ and $J_{f} > 0$ a.e.~in $\Omega$ we have $J_{f_n} \to J_f$ strongly in $\LL^{1}( \Omega )$
and $\D f_{n}( \cdot )^{-1} \to \D f( \cdot )^{-1}$ in $\mathscr{L}^2$-measure on $\Omega$. It follows then from Vitali's convergence
theorem that
\begin{equation}\label{extractionpoint2}
  \int_{\Lambda} \! \phi (w) \langle \nu_{w},\Phi \rangle \, \dd m(w) = \int_{\Omega} \! \phi (f(z)) \Phi \bigl( \D f(z)^{-1} \bigr) J_{f}(z) \, \dd m(z)
\end{equation}
Fix $\Phi \in \CC_{0}( \M )$ and let $N_{\Phi}$ denote the set of points $w \in \Lambda$ that are not Lebesgue points for $\langle \nu_{\cdot}, \Phi \rangle$,
or where $f^{-1}$ is not differentiable, or where $f^{-1}$ is differentiable and $J_{f^{-1}}(w)=0$. By the Lebesgue differentiation theorem,
Gehring-Lehto theorem and the previously mentioned result from \cite{HK} we have $\mathscr{L}^{2}(N_{\Phi})=0$.

Let $\bigl( \rho_{\varepsilon} \bigr)_{\varepsilon > 0}$ be a standard smooth and radial mollifier on $\R^2$. Take for $w_{0} \in \Lambda \setminus N_{\Phi}$
$\phi = \rho_{\varepsilon}(w_{0}- \cdot )$ for $\varepsilon \in (0, \mathrm{dist}(w_{0}, \partial \Lambda ))$ in \eqref{extractionpoint2} and write
$z_{0}=f^{-1}(w_{0})$ to get
$$
\langle \nu_{w_0}, \Phi \rangle = \lim_{\varepsilon \to 0^{+}} \int_{\Omega} \! \rho_{\varepsilon}\bigl( f(z_{0})-f(z) \bigr)
\Phi \bigl( \D f(z)^{-1} \bigr) J_{f}(z) \, \dd m(z).
$$
Our definition of $N_\Phi$ implies that $f$ is differentiable at $z_0$ and that $J_{f}(z_{0})>0$. Put $A_{0}=\D f(z_{0})$ and note that
\begin{eqnarray*}
  \lim_{\varepsilon \to 0^{+}} \int_{\Omega} \! \rho_{\varepsilon} \bigl( f(z_{0})-f(z) \bigr) \, \dd m(z) &=&
  \lim_{\varepsilon \to 0^{+}} \int_{\R^2} \! \rho_{\varepsilon} \bigl( A_{0}(z-z_{0})+ \mathrm{o}(z-z_{0}) \bigr) \, \dd m(z)\\
  &=& \lim_{\varepsilon \to 0^{+}} \int_{\R^2} \! \rho_{\varepsilon} \bigl( w+ \mathrm{o}( A_{0}^{-1}w) \bigr) J_{A_0}^{-1} \, \dd m(w) = J_{A_0}^{-1}.
\end{eqnarray*}
But then standard results from integration theory yield that
$$
\langle \nu_{w_0}, \Phi \rangle = \Phi \bigl( A_{0}) = \Phi \bigl( \D (f^{-1})(w_{0}) \bigr)
$$
and because the space $\CC_{0}( \M )$ is separable in the sup-norm this concludes the proof.
\end{proof}

\begin{proof}[Proof of Proposition~\ref{superharmonic}]
Fix a principal map $f \colon \C \to \C$ of class $\WW^{1,1}_{\loc}( \C )$ and use the notation of Lemma \ref{karitrick}. For $r>1$ we
consider the dilated map $g_{r}(z) \equiv g(rz)$, $z \in \C$. It has the same properties as $g$, but is additionally conformal
for all $|z| > 1/r$. To avoid complicated notation we shall in the following denote this dilated version simply as $g$ again and likewise
we write $h$ for $h_r$. Also denote $A = A_{g}= (rb_{0},b_{1}/r)$.
We now employ \cite[Cor.~1.2]{IKO} (see also \cite{HP18,Campbell17,DePPra20} and the references therein) to the inverse map $g^{-1}$, which
in view of \eqref{l1distort} is a $\WW^{1,2}$-homeomorphism of $g( \overline{\DD} )$ onto $\overline{\DD}$ and it is diffeomorphic close
to the boundary $\partial g( \DD )$.
Hereby we find a sequence $( g_{n})$ of $\CC^{1}$-diffeomerphisms $g_{n} \colon \DD \to g( \DD )$ with $g_{n}^{-1} = g^{-1}$ on $\partial g( \DD )$,
$g_{n}^{-1} \to g^{-1}$ uniformly on $g( \overline{\DD})$ and strongly in $\WW^{1,2}( g( \DD ))$. Using Proposition \ref{upgradeapprox} we
infer that also $g_{n} \to g$ uniformly on $\overline{\DD}$ and strongly in $\WW^{1,1}( \DD )$. Clearly also $g_{n}=g$ on $\partial \DD$.
Put $\mathcal{E} \equiv A( \overline{\DD})$ and note it is an ellipse.
If we extend each $g_n$ to $\C$ by $h \circ A_{f}$ we obtain Lipschitz maps with the property that $g_{n} \circ A_{f}^{-1}=h$ is conformal
on $\C \setminus \mathcal{E}$ with Lipschitz derivative $h^{\prime}$ and so $h^{\prime}(z) = 1 + \mathrm{O}(z^{-3})$ as $|z| \to \infty$.
Because ${\bf E}$ as mentioned before must be locally Lipschitz we may use a simple change of variables to see, using the asymptotics
of $h$, that ${\bf E}( \D g_{n}) - {\bf E}(A)$ is integrable over $\C$ and proceed similarly to \cite[Lemma 7.2]{AFGKKpreprint}
to obtain, for each fixed $n \in \N$, from the assumed $\WW^{1,\infty}$-quasiconvexity at $A=A_{g_{n}}$ that
$$
\int_{\C} \! \bigl( {\bf E}( \D g_{n}) - {\bf E}(A) \bigr) \, \dd m(z) \geq 0.
$$
We continue with
\begin{eqnarray*}
  \int_{\DD} \! \bigl( {\bf E}( \D g_{n} )-{\bf E}( A) \bigr) \, \dd m(z) &\geq& -
  \int_{\C \setminus \DD} \! \bigl( {\bf E}(h^{\prime}(Az)A) - {\bf E}(A) \bigr) \, \dd m(z)\\
  &=& -\frac{1}{J_A}\int_{\C \setminus \mathcal{E}} \! \bigl( {\bf E}(h^{\prime}(z)A) - {\bf E}(A) \bigr) \, \dd m(z),
\end{eqnarray*}
where we note that the last integrand is integrable and, by assumption \eqref{sh} is also superharmonic over the complement $\C \setminus \mathcal{E}$. 
But the complement of an ellipse is by a result of Sakai \cite{Sakai} a null quadrature domain, so the integral of an integrable
superharmonic function there is non-positive, hence we arrive at
$$
\int_{\DD} \! \bigl( {\bf E}( \D g_{n} )-{\bf E}( A) \bigr) \, \dd m(z) \geq 0.
$$
In order to pass to the limit $n \to \infty$ here we must invoke the growth condition \eqref{supgrow}. First recall \eqref{l1distort} and using
that $g_{n}( \DD ) = g( \DD )$ we get by the convergence properties of $g_n$:
$$
\int_{\DD} \! K_{g_n} \, \dd m(z) = \int_{g ( \DD )} \! | \D (g_{n}^{-1})|^{2} \, \dd m(z) \to \int_{g( \DD )} \! | \D (g^{-1})|^{2} \, \dd m(z)
= \int_{\DD} \! K_{g} \, \dd m(z).
$$
Since also $K_{g_n} \to K_{g}$ in $\mathscr{L}^2$ measure it follows that $K_{g_n} \to K_g$ in $\LL^{1}( \DD )$. In particular, the sequence
$( K_{g_n})$ is equi-integrable on $\DD$. For the Jacobians $J_{g_n}$ we note that $g_{n}( \DD_{r}) = g( \DD_{r})$ for all $r \geq 1$, hence the area formula gives
$\int_{\DD_{r}} \! J_{g_n} \, \dd m(z) = \mathscr{L}^{2}( g( \DD_{r} ))$, so $( J_{g_n} )$ is bounded in $\LL^{1}_{\loc}( \C )$. But then $( J_{g_n}^{\alpha})$ is
equi-integrable on $\DD$ by de la Vall\'{e}e Poussin's criterion.

We turn to the term $\log \bigl( 1+1/J_{g_n}\bigr)$ and remark that as a $\WW^{1,1}$-homeomorphism of integrable distortion the map
$g$ has $J_{g} > 0$ a.e. (see \cite{HK}), so clearly $\log \bigl( 1+1/J_{g_n}\bigr) \to \log \bigl( 1+1/J_{g}\bigr)$ in $\mathscr{L}^2$
measure on $\DD$. Even though $g_{n}=g$ on $\partial \DD$ so that the local bound from \cite{KO,KOR} can be improved to a global bound
\cite[Prop.~12.2]{AFGKKpreprint}, and in fact that we have a very precise bound resulting from $\WW^{1,1}$-principal quasiconvexity of
the functional ${\bf W}$ from Theorem \ref{isochoric-volumetric}, we only get that the sequence $\bigl( \log \bigl( 1+1/J_{g_n}\bigr) \bigr)$
is bounded in $\LL^{1}( \DD )$. We must however exclude concentration in order to proceed with the limit, so the bounds must be upgraded.

To that end note that if $\Psi \colon [0, \infty ) \to [0, \infty )$ is an increasing and continuous function, then we get by elementary
calculus since $g_{n}( \DD ) = g( \DD )$ that
$$
\int_{\DD} \! \Psi \left( \frac{1}{J_{g_n}} \right) J_{g_n} \, \dd m(z) = \int_{g( \DD )} \! \Psi \bigl( J_{g_{n}^{-1}} \bigr) \, \dd m(z).    
$$
If $\Theta \colon [0,\infty ) \to [0,\infty )$ is an increasing, continuous function with $\Theta (t)/t \to 1$ as $t \to 0^{+}$
and $\Theta (t) \to \infty$ as $t \to \infty$, and we put $\Psi (t) = \Theta (t)t$ above then we find
\begin{equation}\label{orlicz}    
\int_{\DD} \! \Theta \left( \frac{1}{J_{g_n}} \right) \, \dd m(z) = \int_{g( \DD )} \! \Theta \left( J_{g_{n}^{-1}} \right) J_{g_{n}^{-1}} \, \dd m(z).
\end{equation}
Recall that $g_n$ equals $g$ outside $\DD$, that $g$ is smooth away from $\DD$ and that $g ( \C ) = \C$. Take a cube $Q$ centered
at $0$ and with axes parallel to the coordinate axes such that $g( \DD ) \subset Q$. We now employ \cite[Th.~2]{GrIwMo95}. Accordingly, if
$\theta \colon [0,\infty ) \to [0,\infty )$ is increasing, continuous $\theta (t)/t \to 1$ as $t \to 0^{+}$ and $\theta (t) \to \infty$ as $t \to \infty$,
then there exists a constant $C >0$ such that with
$$
\Psi (t) = \theta (t)t + t\int_{0}^{t} \! \frac{\theta (s)}{s} \, \dd s, \, t \geq 0,
$$
we have
\begin{equation}\label{gim}
\int_{Q} \! \Psi \bigl( J_{g_{n}^{-1}} \bigr) \, \dd m(z) \leq C\int_{2Q} \! \theta \bigl( | \D (g_{n}^{-1})|^{2} \bigr) | \D (g_{n}^{-1})|^{2} \, \dd m(z)
\end{equation}
for all $n \in \N$. Because $\bigl( \D (g_{n}^{-1}) \bigr)$ is $2$-equi-integrable it follows from the de la Vall\'{e}e-Poussin criterion that
we can chose a function $\theta$ as above and such that the right-hand side of \eqref{gim} is bounded for $n \in \N$.
Put
$$
\Theta (t) \equiv \int_{0}^{t} \! \frac{\theta (s)}{s} \, \dd s = \int_{0}^{\log (1+t)} \! \theta \bigl( \mathrm{e}^{r}-1 \bigr) \frac{\mathrm{e}^r}{\mathrm{e}^{r}-1} \, \dd r
$$
and note that $\Theta$ hereby is a continuous, increasing function with $\Theta (0)=0$ and
$$
\frac{\Theta (t)}{\log (1+t)} \to \infty \, \mbox{ as } \, t \to \infty .
$$
Consequently, if we define
$$
\psi (s) \equiv \int_{0}^{s} \! \theta \bigl( \mathrm{e}^{r}-1 \bigr) \frac{\mathrm{e}^r}{\mathrm{e}^{r}-1} \, \dd r, \, s \geq 0,
$$
then $\psi \colon [0,\infty ) \to [0,\infty )$ is continuous, increasing, $\psi (0)=0$ and $\psi (s)/s \to \infty$ as $s \to \infty$.
Since $\psi \bigl( \log (1+t) \bigr) t = \Theta (t)t \leq \Psi (t)$ and we use $\Psi$ in \eqref{orlicz} we arrive at
$$
\int_{\DD} \! \psi \left( \log \left( 1+\frac{1}{J_{g_n}} \right) \right) \, \dd m(z) \leq
C\int_{2Q} \! \Theta \bigl( | \D (g_{n}^{-1})|^{2} \bigr) \bigl| \D (g_{n}^{-1}) \bigr|^{2} \, \dd m(z),
$$
where the right-hand side as observed in \eqref{gim}  is bounded for $n \in \N$. De la Vall\'{e}e-Poussin yields equi-integrability of the sequence
$\bigl( \log \bigl( 1+1/J_{g_n}\bigr) \bigr)$ on $\DD$. Taken together with the previous observations we conclude that the sequence of positive parts
$\bigl( {\bf E}( \D g_{n})^{+} \bigr)$ is equi-integrable on $\DD$ and since it also converges in measure to ${\bf E}( \D g)^{+}$ it follows
from Vitali's convergence theorem that
$$
\int_{\DD} \! {\bf E}( \D g_{n})^{+} \, \dd m(z) \to \int_{\DD} \! {\bf E}( \D g)^{+} \, \dd m(z).
$$
For the sequence $\bigl( {\bf E}( \D g_{n})^{-} \bigr)$ of negative parts we can estimate using Fatou's lemma, and consequently we have shown that
$$
\fint_{\DD} \! {\bf E}( \D g) \, \dd m(z) \geq {\bf E}(A).
$$
Because ${\bf E}$ is locally Lipschitz we can take $r \to 1^{+}$ to conclude the proof.
\end{proof}

\begin{remark}
The interested reader will see that the proof of Proposition \ref{superharmonic} can give slightly more general results.
We also note that the argument leading to the equi-integrability of the sequence $\bigl( \log (1+1/J_{g_n}) \bigr)$ can be upgraded
to give an Orlicz space version of the estimates in \cite{KO} and we will present that elsewhere \cite{AFGKKtoappear}.
\end{remark}
  
We now give a larger, more general family of examples,  with different symmetries.  In order to motivate these examples we recall that
the profound implications of the restricted quasiconvexity of the Burkholder functional  on the higher integrability theory of quasiconformal
maps led the authors of \cite{AIPS15b} to define and investigate a more general family of functionals, so-called complex Burkholder functionals.
These functionals describe not only the integrability properties of quasiconformal mappings, loosely speaking corresponding to the stretch of
such maps, but also the rotation of such maps. 
While its definition is a bit cumbersome it is very natural from the quasiconformal view point, and we now describe it in some detail.

It was discovered in  \cite{AIPS12,AIPS15b} that a fruitful way to understand the Burkholder functional is to see $\int_{\DD} \B_p(\D f)\,\dd m(z)$ as a weighted
$\LL^p$ norm of $\D f$, where the weight depends on $p$ and $f$:  indeed, first recall that any $K$-quasiconformal map is a $\WW^{1,2}_{\loc}$-solution to the Beltrami equation
$$
f_{\overline z} = \mu(z) f_z, 
$$
where the complex dilatation $\mu = \mu_f$ satisfies $|\mu_f(z)| \leq \frac{K-1}{K+1} < 1$ for a.e. $z$. Further, note that for such a map we may write
$$
\B_{p}(\D f)=\Big(\frac{ p |\mu_f|}{1+|\mu_f|}-1\Big)\Bigl|(1+|\mu_f|) f_z\Bigr|^p.
$$
In order to generalize this definition to describe the optimal complex powers for quasiconformal maps we take the exponent $\mathtt{p}$
to be a complex number which satisfies
\begin{equation}
\label{eq:prange}
1\leq |\mathtt{p}-1|, \qquad \mathtt p\neq 0,
\end{equation}
 and we define $\beta=\beta(\mathtt{p})$ as the complex number determined uniquely through the equations
\begin{equation*}
\label{eq:defbeta}
|\beta|+|\beta-2|=2|\mathtt p-1| \quad \text{ and } \quad \Re(\beta/\mathtt{p})=1.
\end{equation*}
Then we set
\begin{equation*}
\label{eq:BpC}
\B_{\mathtt{p}}(A)=\Big( \frac{\mathtt{p}\eta |\mu_A|}{1+\eta |\mu_A|}-1\Big)|(a_++\eta |\mu_A|a_+)^\beta|.
\end{equation*}
Here $\eta$ depends on $|\mu_A|$ and $\mathtt{p}$ through the implicit relations
\begin{equation*}
\label{eq:defrho}
|\eta(z)|=1 \quad \text{ and } \quad \tp{arg}(\mathtt{p} \eta)=\tp{arg}(1+\eta|\mu_A|).
\end{equation*}
We readily see that the structure of the complex Burkholder functional is similar to that of $\B_p$ but that it contains complex powers of $f_z$.
Thus, as mentioned above,  its quasiconvexity relates not only to integrabilty issues but also to how fast quasiconformal maps can wind a line into a spiral.
Concerning the structure of the functional we note that, for a complex parameter $\beta$ as above, the functional is $(\Re \beta)$-homogeneous and the
group of symmetries is naturally described by the logarithmic spiral $S_\beta=\{ z: |z^\beta|=1\}$.  We invite the reader to find a function $b\colon \mathbb{R} \to \C$ such that
$\B_{\mathtt{p}}(A)= |a_+^\beta| b(|\mu_A|)$. In a companion paper \cite{AFGKKtoappear} we prove that the full family of local complex Burkholder functionals is principal
quasiconvex. 

\begin{theorem}
Let $\B_{\mathtt{p}}$ be the $\mathtt{p}$-Burkholder functional with the complex exponent $\mathtt{p}$ satisfying condition \eqref{eq:prange}.  Then
$\bigl\{ \B_{\mathtt{p}} \leq 0 \bigr\} = Q_{2}\left( \frac{| \mathtt{p}-1|+1}{| \mathtt{p}-1|-1} \right)$ and
$$
\B_{\mathtt{p}}^{\loc}(A) \equiv \left\{ \begin{array}{cl}
  \B_{\mathtt{p}}(A) & \mbox{ if } \B_{\mathtt{p}}(A) \leq 0,\\
  +\infty & \mbox{ if } \B_{\mathtt{p}}(A) > 0,
\end{array}
\right.
$$
is $\WW^{1,1}$-principal quasiconvex.
\end{theorem} 
We also intend to discuss semicontinuity and existence results for the complex Burkholder functionals and the corresponding
adapted versions of Proposition~\ref{superharmonic} in \cite{AFGKKtoappear}.

\section{Proofs of semicontinuity and existence results}\label{sec:Proofs}



\subsection{Lower semicontinuity along sequences with asymptotically finite distortion}
We begin by giving the proof of a more general version of Theorem \ref{asfinite-lsc1}. 
The strategy is as follows: first we apply a general semicontinuity result for Young measures \cite{Balder}.  Then we localize and show that our assumed growth conditions and
principal quasiconvexity imply that the integrands satisfy the Jensen inequality for the corresponding homogeneous gradient Young measures. 
\begin{theorem}\label{asfinite-lsc}
Let $\mathbf{E} \colon \Omega \times \R^{2 \times 2} \to \R \cup \{ +\infty \}$ be a normal integrand. Assume that  
\begin{equation}\label{mysterious}
\bigl| \mathbf{E}(z,A) \bigr| \leq \mathbf{C}(z,K_{A})\bigl( |A|^{2}+1 \bigr) \quad \forall (z,A) \in \Omega \times  \M_{+} ,
\end{equation}
where $\mathbf{C} \colon \Omega \times [1,+\infty ) \to [1,+\infty )$ is Borel measurable and $\mathbf{C}(z, \cdot ) \colon [1,\infty ) \to [1,\infty )$ is an increasing
function for each $z \in \Omega$, and that for each $z \in \Omega$, $\mathbf{E}(z, \cdot )$ is $\WW^{1,2}$-principal quasiconvex.

Let $( u_{j} )$ be a sequence in $\WW^{1,2}_{\loc}( \Omega )$ and suppose that the maps have asymptotically finite distortion on a measurable subset $\A$ of $\Omega$:
\begin{equation}\label{asfindist}
K \equiv \limsup_{j \to \infty} \mathbf{K}( \D u_{j}) < +\infty \quad \Leb^{2} \mbox{ a.e.~in } \A .
\end{equation}
Then, if $u_{j} \weak u$ in $\WW^{1,2}_{\loc}( \Omega )$ and
\begin{equation}\label{equiint-energy}
\bigl( \mathbf{E}( \cdot ,\D u_{j})^{-} \bigr) \, \mbox{ is equi-integrable on } \A ,
\end{equation}
we have
$$
\liminf_{j \to \infty} \int_{\A} \! \mathbf{E}( \cdot , \D u_{j}) \, \dd m(z) \geq \int_{\A} \! \mathbf{E}( \cdot ,\D u) \, \dd m(z).
$$
\end{theorem}

\begin{remark}\label{asfinite-bitelsc}
The conclusion of Theorem \ref{asfinite-lsc} fails in general without the assumption \eqref{equiint-energy}, but can be reinstated
if formulated in terms of biting convergence as in \cite{IwaniecGehring} and \cite{FuMoSb}. Under the assumptions of Theorem \ref{asfinite-lsc}
except \eqref{equiint-energy} we have that each subsequence of $( u_j )$ admits a further subsequence $(u_{j_k})$ such that
$$
b \ast \! \lim_{k \to \infty} \mathbf{E}( \cdot ,\D u_{j_k}) \geq \mathbf{E}( \cdot ,\D u) \quad \Leb^2 \mbox{ a.e.~in } \A .
$$
In fact, this is a consequence of Jensen's inequality for $\mathbf{E}$ and the Young measures generated by subsequences of $( \D u_{j_k} )$.
That they hold is part of our proof for Theorem \ref{asfinite-lsc} and we invite the reader to check that their validity do not hinge
on the equi-integrability assumption \eqref{equiint-energy}.
Indeed, if $( \nu_z )_{z \in \Omega}$ is a Young measures generated by a subsequence of $( \D u_{j_k})$, then it is not difficult to see that
$$
b \ast \! \lim_{k \to \infty} \mathbf{E}( \cdot , \D u_{j_k}) \geq \int_{\M} \! \mathbf{E}(\cdot , A) \, \dd \nu_{\cdot}(A) \geq \mathbf{E}( \cdot , \D u)
\quad \Leb^2 \mbox{ a.e.~in } \A .
$$
\end{remark}
The key to the proof is the following localization principle for the Young measures that are generated by the derivatives of maps in a weakly
converging sequence in $\WW^{1,2}_{\loc}$ satisfying the asymptotic distortion condition \eqref{asfindist}. We emphasize that in the subsequent statements
and proofs we consider all functions and maps in terms of their precise representatives.

\begin{proposition}\label{prop:localization}
Let $( u_{j} )$ be a bounded sequence in $\WW^{1,2}_{\loc}( \Omega )$ of maps satisfying \eqref{asfindist}. If $( \D u_{j})$ generates the Young
measure $( \nu_{z} )_{z \in \Omega}$, then for $\Leb^2$ almost all $z \in \A$ the measure $\nu_z$ is a homogeneous $\WW^{1,2}$ gradient Young
measure with support in $Q_{2}\bigl( K(z) \bigr)$.
\end{proposition}

\begin{proof}
Put $M_{j}(z) \equiv \sup_{s \geq j} \mathbf{K}( \D u_s )$, $z \in \Omega$. Then $M_{j} \colon \Omega \to [1,+\infty ]$ is measurable and
\begin{equation}\label{decreasM}
M_{j}(z) \geq M_{j+1}(z) \searrow K(z) \, \mbox{ as } \, j \nearrow \infty \, \mbox{ pointwise in } \, z \in \Omega .
\end{equation}
We recall that $K(z)$ is defined at \eqref{asfindist}, where it is also assumed that $K(z) < +\infty$ holds $\Leb^2$ almost everywhere on $\A$.
Because $K$ is also clearly measurable it follows from Lusin's theorem that it is $\Leb^2$ quasi-continuous on $\A$: for each $\varepsilon > 0$
there exists a measurable subset $G = G_{\varepsilon} \subset \A$ such that $K$ is finite and continuous relative to $G$ and
$\Leb^{2}( \A \setminus G) < \varepsilon$.
Fix an $\varepsilon > 0$ and corresponding set $G=G_{\varepsilon}$. Fix $a>1$ and put $G_{j} = G_{j}^{a} \equiv \{ z \in G : \, M_{j}(z) < aK(z) \}$ and record that
$G_j$ are measurable sets that by virtue of \eqref{decreasM} form an ascending sequence that in the limit $j \to \infty$ exhaust $G$:
\begin{equation}\label{increasG}
G_{j} \subset G_{j+1} \, \mbox{ and } \bigcup_{j \in \N} G_{j} = G .
\end{equation}
We next note that the space $\CC_{0}( \DD \times \M )$ with the supremum norm is separable and that we may find countable
families $\mathcal{F}_1$ in $\CC_{c}^{1}( \DD )$ and $\mathcal{F}_2$ in $\CC_{c}^{1}( \M )$ such that the countable family
$\mathcal{F} \equiv \bigl\{ \eta \otimes \Phi : \, \eta \in \mathcal{F}_{1}, \, \Phi \in \mathcal{F}_{2} \bigr\}$ is total in $\CC_{0}( \DD \times \M )$.
Consequently, if $( \kappa^{i}_{z} )_{z \in \DD}$, $i=1,2$, are two Young measures on $\DD$ and
$$
\int_{\DD} \! \langle \kappa_{z}^{1}, \Phi \rangle \eta (z) \, \dd m(z) =\int_{\DD} \! \langle \kappa_{z}^{2}, \Phi \rangle \eta (z) \, \dd m(z)
$$
holds for all $\eta \otimes \Phi \in \mathcal{F}$, then the Young measures are the same: $\kappa_{z}^1 = \kappa_{z}^2$ for $\Leb^2$ almost all $z \in \DD$.

Because the sequence $( u_j )$ is bounded in $\WW^{1,2}_{\loc}( \Omega )$ the sequence of measures $( | \D u_{j}|^{2} \Leb^{2} \restrict \Omega )$ is bounded in
$\CC_{c}( \Omega )^{\ast}$ and hence by Banach-Alaoglu's compactness theorem admits a weak$\mbox{}^\ast$ converging subsequence there. Extracting such a subsequence
(that we for notational convenience do not relabel) we can assume that
\begin{equation}\label{weakstarlim}
| \D u_{j}|^{2} \Leb^{2} \restrict \Omega \wstar \lambda \, \mbox{ in } \, \CC_{c}( \Omega )^{\ast},
\end{equation}
where $\lambda$ is a positive Radon measure on $\Omega$. In particular we then have for each disk $\DD_{r}(z_{0})$ that is compactly contained in $\Omega$
that
$$
\limsup_{j \to \infty} \fint_{\DD_{r}(z_{0})} \! | \D u_{j}|^{2} \, \dd m(z) \leq \frac{\lambda \bigl( \overline{\DD_{r}(z_{0})}\bigr)}{\Leb^{2}\bigl( \DD_{r}(z_{0})\bigr)} ,
$$
and hence by Lebesgue's differentiation theorem that
\begin{equation}\label{l2bound}
\limsup_{r \searrow 0}\limsup_{j \to \infty} \fint_{\DD_{r}(z_{0})} \! | \D u_{j}|^{2} \, \dd m(z) \leq \frac{\dd \lambda}{\dd \Leb^{2}}(z_{0}) < +\infty 
\end{equation}
holds for $\Leb^2$ almost all $z_{0} \in \Omega$. We now identify a set of \textit{bad} points in $\Omega$ that we would like to avoid in the subsequent argument.
First recall the countable family $\mathcal{F}_2$, and note that for each $\Phi \in \mathcal{F}_2$ the function
$z \mapsto \langle \nu_{z},\Phi \rangle$ is an $\LL^{\infty}( \Omega )$ function, so that $\Leb^2$ almost all $z_{0} \in \Omega$ are Lebesgue points:
$$
\lim_{r \searrow 0} \fint_{\DD_{r}(z_{0})} \! \langle \nu_{z},\Phi \rangle \, \dd m(z) = \langle \nu_{z_{0}},\Phi \rangle \in \R .
$$
Let $N_{\Phi}$ denote the complement in $\Omega$ of this set of Lebesgue points. Next, let $N$ denote the set of points in $\Omega$, where \eqref{l2bound}
fails. Define $\mathcal{N} \equiv N \cup \bigcup_{\Phi \in \mathcal{F}_2} N_{\Phi}$ and record that $\Leb^{2}( \mathcal{N})=0$.

Define for $\DD_{r}(z_{0}) \Subset \Omega$ the maps
$$
v_{j,r}(z) \equiv \frac{u_{j}(z_{0}+rz)-(u_{j})_{z_{0},r}}{r}, \quad z \in \DD ,
$$
where $(u_{j})_{z_{0},r}$ is the integral mean of $u_j$ over $\DD_{r}(z_{0})$. Hereby $v_{j,r} \in \WW^{1,2}( \DD )$ and $(v_{j,r})_{0,1}=0$. Before the next display
we introduce some convenient notation for functions $\theta \colon \C \to \R$: the reflected function, $\tilde{\theta}$, is defined as $\tilde{\theta}(z) \equiv \theta (-z)$,
and the $\LL^1$-dilated function with factor $r>0$, $\theta_{r}$, is defined as $\theta_{r}(z) \equiv \theta \bigl( z/r \bigr)/r^2$. In these terms we
note that for each $z_{0} \in \Omega \setminus \mathcal{N}$ and $\eta \otimes \Phi \in \mathcal{F}$ the above Lebesgue point property and a standard result
about convolution yield (here $\eta$ is extended to $\C \setminus \DD$ by $0$):
\begin{eqnarray*}
  \lim_{r \searrow 0} \lim_{j \to \infty} \int_{\DD} \! \eta \Phi ( \D v_{j,r}) \, \dd m(z) &=& \lim_{r \searrow 0} \lim_{j \to \infty} \int_{\Omega} \! \tilde{\eta}_{r}(z_{0}-z)\Phi ( \D u_{j}) \, \dd m(z)\\
  &=& \lim_{r \searrow 0} \int_{\Omega} \! \tilde{\eta}_{r}(z_{0}-z) \langle \nu_{z},\Phi \rangle \, \dd m(z)\\
  &=& \langle \nu_{z_{0}}, \Phi \rangle \int_{\DD} \! \eta \, \dd m(z) .
\end{eqnarray*}
The bound \eqref{l2bound} translates to
$$
\limsup_{r \searrow 0} \limsup_{j \to \infty} \int_{\DD} \! | \D v_{j,r}|^{2} \, \dd m(z) \leq \pi \frac{\dd \lambda}{\dd \Leb^{2}}(z_{0}) < +\infty .
$$
Fix a point $z_{0} \in G \setminus \mathcal{N}$ that is also a density point for $G$. In addition to the above properties we now also have the following
double limit:
$$
\lim_{r \searrow 0} \lim_{j \to \infty} \frac{\Leb^{2}\bigl( \DD_{r}(z_{0}) \cap G_{j} \bigr)}{\Leb^{2}\bigl( \DD_{r}(z_{0})\bigr)} = 1.
$$
In order to choose a good null sequence $r_j \searrow 0$ we employ a diagonalization argument in connection with the above double limits: Hereby
we find for $v_{j} \equiv v_{j,r_{j}}$ that
\begin{equation}\label{c1}
\lim_{j \to \infty} \int_{\DD} \! \eta \Phi ( \D v_{j}) \, \dd m(z) = \int_{\DD} \! \eta \, \dd m(z) \langle \nu_{z_{0}},\Phi \rangle \quad \forall \, \eta \otimes \Phi \in \mathcal{F},
\end{equation}
\begin{equation}\label{c2}
\sup_{j \in \N} \int_{\DD} \! | \D v_{j}|^{2} \, \dd m(z) < +\infty
\end{equation}
and
\begin{equation}\label{c3}
\lim_{j \to \infty} \frac{\Leb^{2}\bigl( \DD_{r_j}(z_{0}) \cap G_{j} \bigr)}{\Leb^{2}\bigl( \DD_{r_j}(z_{0})\bigr)} = 1.
\end{equation}
Because $(v_{j})_{0,1}=0$ we deduce using Poincar\'{e}'s inequality and \eqref{c2} that the sequence $(v_{j})$ is bounded in $\WW^{1,2}( \DD )$. Therefore
any subsequence of $( \D v_j )$ admits a further subsequence that generates a Young measure. By \eqref{c1} this Young measure is the homogeneous
Young measure $( \nu_{z_0} )_{z \in \DD}$. A routine argument then shows that the full sequence $( \D v_j )$ generates the homogeneous Young measure
$( \nu_{z_0})_{z \in \DD}$. It is then easy to check that $( v_j )$ converges weakly in $\WW^{1,2}( \DD )$ to the linear map $\langle \nu_{z_0},\mathrm{Id} \rangle$.
In particular it follows that $( \nu_{z_0})_{z \in \DD}$ is a homogeneous $\WW^{1,2}$ gradient Young measure. We use \eqref{c3} to get information about the support
of $\nu_{z_0}$. Recall the parameter $a>1$ used in the definition of $G_j$ and put
$$
\mathcal{E}=\mathcal{E}_{a} \equiv \bigl\{ A \in \R^{2 \times 2}: \, |A|^{2} > a^{2}K(z_{0})\det A \bigr\} .
$$
This is clearly an open set in $\M$ and we note that for $z \in \DD$ we have $\D v_{j}(z) \in \mathcal{E}$ precisely when
$| \D u_{j} (z_{0}+r_{j}z) |^{2} > a^{2}K(z_{0})\det \D u_{j}(z_{0}+r_{j}z)$. Now recall that $K$ is continuous and finite relative to $G$ and so for some $j_a \in \N$
we have $K(z_{0}+r_{j}z) < aK(z_{0})$ for all $z \in \DD \cap \bigl( (G-z_{0})/r_{j} \bigr)$ for $j \geq j_a$. In view of the definition of $G_j$ we
therefore have when $j \geq j_a$ that
\begin{eqnarray}
  | \D u_{j}(z_{0}+r_{j}z)|^{2} &\leq& aK(z_{0}+r_{j}z) \det \D u_{j}(z_{0}+r_{j}z)\nonumber\\
  &\leq& a^{2}K(z_{0}) \det \D u_{j}(z_{0}+r_{j}z)\label{estionG}
\end{eqnarray}
holds for all $z \in \DD \cap \bigl( (G_{j}-z_{0})/r_{j} \bigr)$.
Because $\mathcal{E}$ is an open set its indicator function $\mathbf{1}_{\mathcal{E}}$ is lower semicontinuous, so we can find functions
$\Phi_{k} \in \CC_{c}( \R^{2 \times 2})$ such that $0 \leq \Phi_{k} \leq \Phi_{k+1} \nearrow \mathbf{1}_{\mathcal{E}}$ as $k \nearrow +\infty$. Now for each $k \in \N$,
$$
\liminf_{j \to \infty} \fint_{\DD} \! \mathbf{1}_{\mathcal{E}}( \D v_{j}) \, \dd m(z) \geq
\liminf_{j \to \infty} \fint_{\DD} \! \Phi_{k}( \D v_{j}) \, \dd m(z) = \langle \nu_{z_0} , \Phi_{k} \rangle
$$
and so taking $k \nearrow +\infty$ we get by virtue of Lebesgue's monotone convergence theorem
$$
\liminf_{j \to \infty} \fint_{\DD} \! \mathbf{1}_{\mathcal{E}}( \D v_{j}) \, \dd m(z) \geq \nu_{z_{0}}\bigl( \mathcal{E} \bigr) .
$$
To conclude the proof we estimate the left-hand side using \eqref{estionG} for $j \geq j_a$:
\begin{eqnarray*}
  \fint_{\DD} \! \mathbf{1}_{\mathcal{E}}( \D v_{j}) \, \dd m(z) &=& \frac{1}{\pi}\left( \int_{\DD \cap (G_{j}-z_{0})/r_{j}} + \int_{\DD \setminus (G_{j}-z_{0})/r_{j})} \right)
  \mathbf{1}_{\mathcal{E}}( \D v_{j}) \, \dd m(z)\\
  &\leq& \frac{\Leb^{2}\bigl( \DD_{r_{j}}(z_{0})\setminus G_j \bigr)}{\Leb^{2}\bigl( \DD_{r_{j}}(z_{0})\bigr)}
\end{eqnarray*}
and the latter tends by \eqref{c3} to $0$ as $j \to \infty$. Consequently, $\nu_{z_0}\bigl( \mathcal{E} \bigr) =0$. Recall that $\mathcal{E}=\mathcal{E}_{a}$
and that $a>1$ is arbitrary, so by monotone convergence we deduce that $\nu_{z_0}\bigl( \bigcup_{a > 1}\mathcal{E}_{a} \bigr)=0$.
But then $\nu_{z_0}$ is carried by $Q_{2}(K(z_{0}))$, and since $Q_{2}(K(z_{0}))$ is a closed set it must contain the support of $\nu_{z_0}$,
as required. The above holds for all $z_0 \in G \setminus \mathcal{N}$ that are density points of $G$, where we recall that $G$ was a measurable subset of
$\A$ such that $\Leb^{2}( \A \setminus G) < \varepsilon$. Here $\varepsilon > 0$ was arbitrary and it is then routine to conclude the proof.
\end{proof}

\begin{proof}[Proof of Theorem \ref{asfinite-lsc}.]
Because of \eqref{equiint-energy} we can without loss of generality assume, considering a suitable subsequence that we for notational convenience
do not relabel, that
$$
\lim_{j \to \infty}\int_{\A} \mathbf{E}( \cdot , \D u_{j}) \, \dd m(z) = \liminf_{j \to \infty}\int_{\A} \mathbf{E}( \cdot ,\D u_{j}) \, \dd m(z) \in \R
$$
and that the sequence $( \D u_{j})$ generates the Young measure $( \nu_{z})_{z \in \Omega}$. Here, for each $z \in \Omega$, the partial functional
$\mathbf{E}(z, \cdot )$ is lower semicontinuous on $\M_{+}$, and so it coincides with its lower semicontinuous envelope
$$
\mathbf{E}_{\ast}(z,A) \equiv \liminf_{A^{\prime} \to A} \mathbf{E}(z,A^{\prime})
$$
there. Note that $\mathbf{E}_{\ast} \colon \Omega \times \M \to [-\infty , +\infty ]$ is Borel measurable and that
$$
\lim_{j \to \infty}\int_{\Omega} \mathbf{E}( \cdot , \D u_{j}) \, \dd m(z) = \lim_{j \to \infty}\int_{\Omega} \mathbf{E}_{\ast}( \cdot , \D u_{j}) \, \dd m(z)  \in \R ,
$$
where again the negative parts $\bigl( \mathbf{E}_{\ast}( \cdot , \D u_{j})^{-} \bigr)$ form an equi-integrable sequence on $\A$.
We may then appeal to a general lower semicontinuity result from Young measure theory \cite{Balder} whereby
$$
\lim_{j \to \infty}\int_{\A} \mathbf{E}_{\ast}( \cdot ,\D u_{j}) \, \dd m(z) \geq \int_{\A}^{\ast} \! \langle \nu_{z}, \mathbf{E}_{\ast}(z, \cdot ) \rangle \, \dd m(z).
$$
It is at this stage we employ the localization principle from Proposition \ref{prop:localization}. Accordingly $\nu_z$ is, for $\Leb^2$ almost all $z \in \A$, a homogeneous
$\WW^{1,2}$ gradient Young measure supported on $Q_{2}\bigl( K(z) \bigr)$. Because $\mathbf{E}(z, \cdot )$ in particular is lower semicontinuous there we infer that
$$
\langle \nu_{z}, \mathbf{E}_{\ast}(z, \cdot ) \rangle = \langle \nu_{z}, \mathbf{E}(z, \cdot ) \rangle \quad \mbox{ for $\Leb^2$ a.e. } z \in \A .
$$
Moreover, $\mathbf{E}(z, \cdot )$ is principal quasiconvex and real-valued on $\M_{+}$, so must be rank-one convex and thus locally Lipschitz on $\M_{+}\setminus \{0\}$.
It need not be continuous at $0$ relative to $\M_{+}$, but is lower semicontinuous there by assumption.

Next, for $\mathscr L^2$ a.e.\ $z\in \mathbb A$, from \cite{AstalaFaraco02} we find a sequence of principal $K(z)$-quasiconformal maps
$f_{j} \colon \C \to \C$ such that $f_{j} \weak \langle \nu_{z}, \mathrm{Id} \rangle$ in $\WW^{1,p}( \DD )$ for each $p<2K(z)/(K(z)-1)$ and $( \D f_{j})$ generates $\nu_z$.
Recall that according to the 0-1 Law for homogeneous gradient Young measures from \cite{AstalaFaraco02} we have that either $\nu_{z}( \{ 0 \} ) =0$
or $\nu_{z}=\delta_{0}$. For measures $\nu_z$ of the former type, $\mathbf{E}(z, \cdot )$ is therefore continuous $\nu_z$ almost everywhere and so as
the quadratic growth \eqref{mysterious} ensures that the sequence $\bigl( \mathbf{E}( \D f_{j} )\bigr)$ is equiintegrable on $\DD$, we arrive at
$$
\fint_{\DD} \! \mathbf{E}(z, \D f_{j}(w)) \, \dd m(w) \to \langle \nu_{z},\mathbf{E}(z, \cdot ) \rangle .
$$
Since we also have that $f_j \weak A_f$ in $\WW^{1,2}( \DD )$, where $\ A_f = \langle \nu_{z},\mathrm{Id} \rangle$,
we get that $A_{f_j} \to A_{f}$.
It now follows directly from principal quasiconvexity  that $\langle \nu_{z},\mathbf{E}(z, \cdot ) \rangle \geq \mathbf{E}\bigl( z, \langle \nu_{z},\mathrm{Id} \rangle \bigr)$ and this
concludes the proof for measures with $\nu_{z}( \{ 0 \} )=0$.

For the other measures $\nu_z = \delta_{0}$ we simply use that $\mathbf{E}(z, \cdot )$ is lower semicontinuous at $0$, that the quadratic growth \eqref{mysterious}
still ensures that the sequence $\bigl( \mathbf{E}( \D f_{j} )\bigr)$ is equiintegrable on $\DD$ (in particular the negative parts) and then
the general lower semicontinuity result for Young measures \cite{Balder}. This concludes the proof.
\end{proof}

\subsection{Lower semicontinuity with integrable distortion}
In this subsection we give the proof of a more general version of Theorem \ref{thm:lscPQ}, which  follows the proof in \cite[Section 12]{AFGKKpreprint}
concerning the lower semicontinuity of the non polyconvex functional $\textbf{W}$ whose definition was recalled at \eqref{wmagic}. This functional is the Shield
transform of a functional introduced in \cite{Tade}, \cite{AIPS12} and it was also studied in the recent works \cite{Voss1,Voss2}.  As the result below shows, the crucial
feature of $\textbf{W}$ is its principal quasiconvexity.
 
\begin{theorem}\label{thm:intdistort-lsc}
Let $\mathbf{E} \colon \Omega \times \M \to \R \cup \{ +\infty \}$ be a Borel measurable functional such that $\mathbf{E}(z, \cdot )$ is $\WW^{1,1}$
principal quasiconvex for each $z \in \Omega$. Assume that there exist constants $C >0$ and $p \in [1,2)$ such that
\begin{equation}\label{eq:growthconditionsnom}
 \bigl| {\bf E}(z, A) \bigr| \leq C \max \bigl\{|A|^{p}, -\log(\det A), K_A \bigr\}+C 
\end{equation}
holds for all $(z,A) \in \Omega \times \M_{+}$, while ${\bf E}(z,A) = +\infty$ for $(z,A) \in \Omega \times \bigl( \R^{2 \times 2}\setminus \bigl( \M_{+} \cup \{ 0 \} \bigr) \bigr)$. 
Let $q>\frac{p}{2-p}$ be an exponent and $g \in \WW^{1,2}_{\loc}( \C )$ a homeomorphism with $K_{g} \in \LL^{q}_{\loc}( \C )$.
If $( u_j )$ is a sequence of homeomorphisms, $u_{j} \colon \Omega \to u_{j}( \Omega )$, $u_{j} \in g+\WW^{1,2}_{0}( \Omega )$ with $\sup_{j} \| K_{u_j} \|_{\LL^{q}( \Omega )} < +\infty$
and $u_{j} \weak u$ in $\WW^{1,2}_{\loc}( \Omega )$, then $u \in g+\WW^{1,2}_{0}( \Omega )$ is a homeomorphism with $K_{u} \in \LL^{q}( \Omega )$ and
$$
\liminf_{j \to \infty} \int_{\Omega} \! \mathbf{E}( \cdot , \D u_{j}) \, \dd m(z) \geq \int_{\Omega} \! \mathbf{E}( \cdot ,\D u) \, \dd m(z).
$$
\end{theorem}
Apart from the use of principal quasiconvexity the overall proof strategy is standard. The implementation is however not and it might have other
applications and so be of wider interest. First we show that the assumed growth conditions suffice to ensure equiintegrability of the energy densities. This relies on
\cite{KOR,KO} as presented in \cite[Proposition 12.8]{AFGKKpreprint}. Next, by Young measure theory the proof then reduces to obtaining Jensen's inequality
for a homogeneous gradient Young measure, which is the content of Proposition~\ref{thm:intdistort}.

\begin{proposition}\label{prop:equi} 
Let ${\bf E}$ satisfy \eqref{eq:growthconditionsnom} and let $(v_j)$ be a bounded sequence in $\WW^{1,2}(\Omega)$ with $\sup_j \|K_{u_j}\|_{\LL^q(\Omega)}<\infty$
for some $q>\frac{p}{2-p}$. Then the sequence $\bigl( {\bf E}( \cdot , \D v_j) \bigr)$ is equiintegrable in $\Omega$.
\end{proposition}

\begin{proof}
The assumptions imply that both $|\D v_j|^{p}$ and $K_{v_j}$ are equiintegrable.
In order to see that $\log J_{v_j}$ is equiintegrable too we invoke the pointwise estimate:
\begin{equation}\label{eq:elementaryestimatelog}
|\log(J_{v_j})| \leq  \log\left(e+\frac{1}{J_{v_j}}\right) + (J_{v_j})^{\frac 1 2}.
\end{equation}
Here the first term on the right-hand side is equiintegrable by virtue of \cite[Proposition 12.8]{AFGKKpreprint}, while equiintegrability of the second follows from
the assumed $\WW^{1,2}$ bound.
\end{proof}

\begin{proof}[Proof of Theorem~\ref{thm:intdistort-lsc}]
It is well-known that the limit $u \in g+\WW^{1,2}_{0}( \Omega )$ is a homeomorphism with distortion $K_{u} \in \LL^{q}( \Omega )$, see \cite[Ch.~21]{AIMbook}.
In view of Proposition \ref{prop:equi} the sequence $\bigl( {\bf E}( \cdot , \D u_{j}) \bigr)$ is equiintegrable on $\Omega$, and for a subsequence that we for
convenience do not relabel we have
$$
\int_{\Omega} \! \mathbf{E}( \cdot , \D u_{j}) \, \dd m(z) \to \ell \in \R .
$$
Extracting a further subsequence, again not relabelled, we may assume that $( \D u_j )$ generates the $\WW^{1,2}$-gradient Young measure $(\nu_z)_{z\in \Omega}$.
First, note that the $\LL^q$ bound entails
$$
\int_{\Omega} \! \langle \nu_{z} , {\bf K} \rangle \, \dd m(z) = \lim_{j \to \infty}\int_{\Omega} \! K_{u_j} \, \dd m(z) < +\infty ,
$$
so that, in particular, $\langle \nu_{z},{\bf K} \rangle < +\infty$ holds for a.e.~$z \in \Omega$. If therefore ${\bf E}_{\ast}(z, \cdot )$ denotes the lower semicontinuous
envelope of ${\bf E}(z, \cdot )$, then ${\bf E}_{\ast}(z, \cdot )={\bf E}(z, \cdot )$ $\nu_{z}$-a.e. for $\mathscr{L}^2$ a.e.~$z \in \Omega$.
Hence we get by general Young measure theory \cite{Balder},
\begin{equation}
\label{eq:Balder2}
\lim_{j\to \infty}\int_\Omega{\bf E}(z,\D u_j)\,\dd m(z) =\ell \geq
\int_\Omega \int_{\R^{2\times 2}}{\bf E}(z,A) \,\dd \nu_z(A) \, \dd m(z) > -\infty .
\end{equation}
For a.e.\ $z\in \Omega$ it is routine to check that we have  
\begin{equation}
\label{eq:conditionsbarycenter}
A_{z} \equiv \langle \nu_{z},\mathrm{Id} \rangle \in \M_{+} \, \mbox{ and } \, \nu_{z} \mbox{ is a homogeneous $\WW^{1,2}$ gradient Young measure.}
\end{equation}
Let us fix a point $z\in \Omega$ for which \eqref{eq:conditionsbarycenter} holds.
The measure $\nu_{z}$ is generated, as a homogeneous Young measure over $\DD$, by taking a suitable diagonal subsequence $(\psi_{j,\lambda_{j}})$ of
$\psi_{j,\lambda}(w)=\lambda^{-1}\bigl( u_{j}(z_0+\lambda w)-u_{j}(z_{0}) \bigr)$, where $j\to \infty$ and $\lambda_j \to 0$, see~\cite[Theorem 2.8]{AFGKKpreprint}.
In particular,  $\psi_{j,\lambda_{j}} \colon \DD \to \C$ is a sequence of homeomorphisms such that $\psi_{j,\lambda_{j}} \to A_z$ weakly in $\WW^{1,2}( \DD )$
and $\sup_{j}\|K_{\psi_{j,\lambda_{j}}}\|_{\LL^q(\DD )} <\infty$. 
Then we are precisely in position to apply Proposition~\ref{thm:intdistort}, whereby we find a sequence $(f_j)$ of principal homeomorphisms of class $\WW^{1,1}_{\loc}( \C )$
satisfying $f_j \to A_{z}$ locally uniformly in $\mathbb{C}$, $f_j \to A_{z}$ weakly in $\WW^{1,2}_{\loc}( \DD )$,  $A_{f_j}=\fint_{\DD} \! \D f_j \, \dd m(z) \to A_{z}$,
$K_{f_j}=K_{\psi_{j,\lambda_{j}}}$ a.e.~in $\DD$ and $( \D f_{j}|_{\DD} )$ generates $\nu_z$. Now as $f_j$ are principal maps it follows that $( J_{f_j})$ is bounded in
$\LL^{1}_{\loc}( \C )$, and therefore we have for $s \in (\frac{p}{2-p},q)$ that
$$
| \D f_{j}|^{p} = J_{f_j}^{\frac{p}{2}}K_{f_j}^{\frac{p}{2}} \leq J_{f_j}^{\frac{ps}{2s-p}}+K_{f_j}^{s}
$$
is equiintegrable on $\DD$. In view of Proposition~\ref{prop:equi} the sequence $\bigl( {\bf{E}}(z,\D f_j(w)) \bigr)$ is equiintegrable over $\DD$
and since ${\bf E}(z, \cdot )$ is principal quasiconvex, it is in particular rank-one convex on $\M_{+}$, and so continuous there, hence
$$
\int_{\M} \! {\bf E}(z,A) \,\dd \nu_{z}(A) = \lim_{j \to \infty} \fint_{\DD} \! {\bf E}(z, \D f_{j}(w)) \, \dd m(w) \geq \lim_{j \to \infty} {\bf E}(z,A_{f_j}) = {\bf E}(z,A_z ),
$$
where the last inequality follows by the assumed principal $\WW^{1,1}$ quasiconvexity.
Inserting the Jensen inequality for each $\nu_z$ in \eqref{eq:Balder2} yields the required lower semicontinuity. 
\end{proof}

\begin{proof}[Proof of Corollary~\ref{existence}]
We start by establishing a coercivity inequality and note that the growth condition \eqref{eq:growthconditions} implies that
${\bf E}(A) \geq -C\bigl( 1+|A|^{p}+K_{A}+|\log J_{A}| \bigr)$ holds for all $A \in \M_{+}$. Consequently, invoking Young's inequality in
a routine manner we find positive constants $c_1$, $c_{2}>0$ such that ${\bf E}(A)+{\bf P}(A) \geq c_{1}|A|^{2}+c_{1}K_{A}^{q}-c_{2} - C| \log J_{A}|$
holds for all $A \in \M_{+}$. Next, combining the bound \eqref{eq:elementaryestimatelog} with \cite[Proposition 12.8]{AFGKKpreprint} as in the proof of
Proposition~\ref{prop:equi} above we arrive at
$$
\int_{\Omega} \! \bigl( {\bf E}( \D u) + {\bf P}( \D u) \bigr) \, \dd m(z) \geq c_{0}\int_{\Omega} \! \bigl( | \D u|^{2}+K_{u}^{q} \bigr) \, \dd m(z) - C_{0}  
$$
for all $u \in \mathcal{A}_{g}$, where $c_0$, $C_0 > 0$ are positive constants that in particular are independent of $u$. Again using Young's inequality
in a standard manner we infer that the energy including the forcing term,
$$
\int_{\Omega} \! \bigl( {\bf E}( \D u)+{\bf P}( \D u) \bigr) \, \dd m(z) + \bigl\langle F,u \bigr\rangle
$$
is bounded below on $\mathcal{A}_{g}$ and that, since ${\bf E}( \D g)+{\bf P}( \D g)$ is integrable over $\Omega$ by assumption, any minimizing sequence
$( u_{j}) \subset \mathcal{A}_g$ is bounded in $\WW^{1,2}( \Omega )$ and has distortions $( K_{u_j})$ bounded in $\LL^{q}( \Omega )$. By general principles
we may then extract a subsequence (for convenience not relabelled) such that $u_{j} \to u$ weakly in $\WW^{1,2}( \Omega )$. It then follows that
$u \in g+\WW^{1,2}_{0}( \Omega )$ and, see \cite[Ch.~21]{AIMbook}, that $u \colon \Omega \to g( \Omega )$ is a homeomorphism with $K_{u} \in \LL^{q}( \Omega )$.
In particular, $u \in \mathcal{A}_{g}$ and  by  \cite{Ball1} and Theorem~\ref{thm:intdistort-lsc} $u$ is therefore a minimizer, as required. 
\end{proof}

\subsection{Lower Semicontinuity of the critical Burkholder energy }

\subsubsection{An integrability property of the Burkholder functional}
Before starting the proof we record the following integrability property of the Burkholder functional. That a result of this type holds is well-known \cite{AIMbook},
but our statement here is more explicit.

\begin{proposition}[$\LL^1$ boundedness] \label{prop:l1bound} Let $p > 2$ and put $K \equiv p/(p-2)$. There exists a constant $C=C(p)$ with the following property.
If $u \colon \Omega \to \C$ is a $K$-quasiregular map and $\DD_{R}(z_{0}) \Subset \Omega$, then for all $\delta\in (0,1)$ we have
\begin{equation}
\fint_{\DD_{\delta R}(z_{0})} \! \B_{p}( \D u) \, \dd m(z) \geq -\frac{C}{R^{pK}(1-\delta )^{pK}}\sup_{\DD_{R}(z_{0})}|u|^{p}.
\end{equation}
\end{proposition}
The point here is, as was mentioned in the Introduction, that a $K$-quasiregular map need not be of class $\WW^{1,p}_{\loc}$. That the Burkholder functional
is locally integrable despite being nonpositive and $p$-homogeneous on the derivative of a $K$-quasiregular map comes down to it vanishing on the boundary of
the $K$-quasiconformal well,  where the large values of the derivative must concentrate.

\begin{proof}
Changing variables we may without loss of generality assume that $z_{0}=0$ and $R=1$. We then Stoilow factorize $u$ on $\DD$, whereby
$u=h \circ f$ for a $K$-quasiconformal normalized principal map $f \colon \C \to \C$ and a holomorphic map $h \colon f(\DD ) \to \C$.
If $m \equiv \sup_{\DD} |u|$, then also $m = \sup_{f(\DD )}| h|$ and so by Cauchy's integral formula and H\"older  properties of $K$-quasiconformal principal
maps \cite{AIMbook} we estimate
$$
\sup_{\delta \DD}\bigl| h^{\prime} \circ f \bigr| \leq \frac{m}{\dist (f( \delta \DD ),\partial f(\DD ))} \leq c_{p}\frac{m}{(1-\delta )^{K}}.
$$
Next, using that $\B_p$ is nonpositive and the Burkholder area inequality \cite{AFGKKpreprint}:
\begin{eqnarray*}
  \int_{\delta \DD} \! \B_{p}( \D u) \, \dd m(z) &=& \int_{\delta \DD} \! \bigl| h^{\prime} \circ f \bigr|^{p} \B_{p}( \D f) \, \dd m(z)\\
  &\geq& \sup_{\delta \DD} \bigl| h^{\prime} \circ f \bigr|^{p} \int_{\delta \DD} \! \B_{p}( \D f) \, \dd m(z)\\
  &\geq& c_{p}^{p}\frac{m^{p}}{(1-\delta )^{pK}}\int_{\DD} \! \B_{p}( \D f) \, \dd m(z)\\
  &\geq& -c_{p}^{p}\frac{m^{p}}{(1-\delta )^{pK}}\pi ,
\end{eqnarray*}
as required.
\end{proof}

\subsection{Proof of Theorem \ref{swlscBurk}} 

\begin{proof}
We split the proof into two steps and remark that the assumption $u \in \WW^{1,p}_{\loc}( \Omega )$ only is used in the second step.
\bigskip

\noindent
By Proposition~\ref{prop:l1bound},  the sequence $\bigl( \B_{p}( \D u_j ) \bigr)$ is bounded in $\LL^{1}_{\loc}( \Omega )$. 
Because $\LL^{1}_{\loc}( \Omega ) \subset \CC_{c}( \Omega )^{\ast}$ continuously, the Banach-Alaoglu compactness theorem
implies that any subsequence of $\bigl( \B_{p} ( \D u_{j}) \bigr)$ admits a further subsequence that converges weakly$\mbox{}^\ast$
to a Radon measure on $\Omega$. Any such limit must necessarily be a nonpositive Radon measure. Assume that $-\lambda$ is such
a weak$\mbox{}^\ast$ limit and consider the Lebesgue-Radon-Nikodym decomposition
$$
\lambda = \frac{\dd \lambda}{\dd \Leb^2}\Leb^{2} + \lambda^{s}, \, \mbox{ where } \lambda^{s} \perp \Leb^{2}.
$$
We next estimate each term in this decomposition.
\bigskip

\noindent
\textbf{Step 1.} $-\frac{\dd \lambda}{\dd \Leb^{2}} \geq \B_{p}( \D u)$ holds $\Leb^2$ almost everywhere in $\Omega$.
The proof consist of a rather general localization argument at a point of differentiabilty of $u$.
Arguments of this sort are known in the literature as blow-up arguments and in the present context go back to I.~Fonseca and S.~M\"{u}ller \cite{FoMu}.
The fact that the Burkholder functional is nonpositive on its effective domain means that one must be more careful than in the classical
set-up, where the considered functionals are bounded below and therefore concentration effects are not so harmful.
\bigskip

Let $z_{0} \in \Omega$ be a point satisfying
\begin{equation}\label{z1}
u \mbox{ is differentiable at } z_{0} \mbox{ with Jacobi matrix } A_{0} \equiv \D u(z_{0})
\end{equation}
and
\begin{equation}\label{z2}
  \frac{\dd \lambda}{\dd \Leb^2}(z_{0}) = \lim_{r \searrow 0} \frac{\lambda (\DD_{r}(z_{0}))}{\Leb^{2}( \DD_{r}(z_{0}))}.
\end{equation}
From the local uniform convergence of $u_j$ to $u$ and \eqref{z1} follows that
$$
\lim_{r \searrow 0} \lim_{j \to \infty} \frac{1}{r}\sup_{z \in \DD_{r}(z_{0})} \bigl| u_{j}(z)-u_{j}(z_{0})-A_{0}(z-z_{0}) \bigr| = 0.
$$
Next, for $r \in (0, \dist (z_{0}, \partial \Omega ))$ with $\lambda (\partial \DD_{r}(z_{0}))=0$ we have
$$
\lim_{j \to \infty} \int_{\DD_{r}(z_{0})} \! \B_{p}( \D u_j ) \, \dd m(z) = -\lambda ( \DD_{r}(z_{0})),
$$
and so by \eqref{z2} we get in particular
$$
\lim_{E \notni r \searrow 0} \lim_{j \to \infty} \fint_{\DD_{r}(z_{0})} \! \B_{p}( \D u_{j}) \, \dd m(z) = -\frac{\dd \lambda}{\dd \Leb^2}(z_{0}),
$$
where $E \equiv \{ r \in (0,\dist (z_{0},\partial \Omega )) : \, \lambda (\partial \DD_{r}(z_{0})) > 0 \}$ is an at most countable set.
Fix $s \in (0,1)$. In view of the above we may choose a null sequence $r_j \searrow 0$ such that in terms of
$$
v_{j}(z) \equiv \frac{u_{j}(z_{0}+r_{j}z)-u_{j}(z_{0})}{r_{j}}, \, z \in \DD,
$$
we have
$$
\lim_{j \to \infty} \int_{\DD_{s}(0)} \! \B_{p}( \D v_j ) \, \dd m(z) = -\pi s^{2}\frac{\dd \lambda}{\dd \Leb^2}(z_{0}).
$$
Here it is clear that $v_j$ are $K$-quasiregular and $v_{j} \to A_0$ uniformly on $\DD$.
The matrix $A_0$ is $K$-quasiconformal, so if its conformal-anticonformal coordinates are $(A_{0}^{+},A_{0}^{-})$, then
$|A_{0}^{-}| \leq |A_{0}^{+}|/(p-1)$. Its Stoilow factorization is therefore (when $A_{0} \neq 0$)
$A_{0} = h \circ f$ with $h(z) \equiv A_{0}^{+}z$ and
$$
f(z) \equiv \left\{
\begin{array}{cl}
  z+\frac{A_{0}^{-}}{A_{0}^{+}}\bar{z} & \mbox{ if } |z| \leq 1,\\
  z+\frac{A_{0}^{-}}{A_{0}^{+}}\frac{1}{z} & \mbox{ if } |z| > 1.
\end{array}
\right.
$$
If $v_j = h_{j} \circ f_j$ is the Stoilow factorization of $v_j$, then using uniqueness of the factorization it is not difficult to see that
$f_j \to f$ uniformly on $\C$ (if $A_0 = 0$ we use normality to extract a convergent subsequence and this then defines $f$) and $h_j \to h$
locally uniformly on $f( \DD )$. Consequently we have that $h_{j}^{\prime}(f_{j}(z)) \to A_{0}^{+}$
uniformly in $z \in \DD_{s}(0)$, hence using the properties of $\B_p$ and in particular the Burkholder area inequality (see \cite{AFGKKpreprint})
and that $A_{f_j} \to (1,\frac{A_{0}^{-}}{A_{0}^{+}})$:
\begin{eqnarray*}
-\pi s^{2} \frac{\dd \lambda}{\dd \Leb^2}(z_{0}) &=& \lim_{j \to \infty} \int_{\DD_{s}(0)} \! \B_{p}( \D v_j ) \,\dd m(z)\\
&=& \lim_{j \to \infty} \int_{\DD_{s}(0)} \! \bigl| h_{j}^{\prime} \circ f_{j} \bigr|^{p} \B_{p}( \D f_{j}) \, \dd m(z)\\
&=& \lim_{j \to \infty} \int_{\DD_{s}(0)} \! |A_{0}^{+}|^{p} \B_{p}( \D f_{j}) \, \dd m(z)\\
&\geq & |A_{0}^{+}|^{p} \limsup_{j \to \infty} \int_{\DD} \! \B_{p}( \D f_{j}) \, \dd m(z)\\
&\geq& |A_{0}^{+}|^{p} \limsup_{j \to \infty} \bigl( \pi \B_{p}(A_{f_j}) \bigr) = \pi \B_{p}(A_{0}).
\end{eqnarray*}
Finally, taking $s \nearrow 1$ we conclude that $-\frac{\dd \lambda}{\dd \Leb^{2}}(z_{0}) \geq \B_{p}( \D u(z_{0}))$ holds
at all points $z_{0} \in \Omega$ satisfying \eqref{z1} and \eqref{z2}. Because this includes $\Leb^2$ almost all points
in $\Omega$ the proof of Step 1 is finished.
\bigskip

\noindent
\textbf{Step 2.} $\lambda^{s} = 0$ when $u \in \WW^{1,p}_{\loc}( \Omega )$.
\bigskip

It suffices to show that $\lambda^{s}(B)=0$ for each disk $B = \DD_{R}( w_0 ) \Subset \Omega$. Fix such a disk $B$.
Let $z_{0} \in B$ satisfy
\begin{equation}\label{z3}
  \lim_{r \searrow 0} \frac{\lambda ( \DD_{r}(z_{0}))}{\Leb^{2}( \DD_{r}(z_{0}))} = +\infty
\end{equation}
and
\begin{equation}\label{z4}
  \lim_{r \searrow 0} \frac{1}{\lambda (\DD_{r}(z_{0}))}\int_{\DD_{r}(z_{0})} \! | \D u|^{p} \, \dd m(z)=0.
\end{equation}
Put
$$
v_{j,r}(z) \equiv \frac{u_{j}(z_{0}+rz)-u_{j}(z_{0})}{r\rho_{r}}, \quad z \in \DD, \mbox{ where }
\rho_{r} \equiv \left(\frac{\lambda ( \DD_{r}(z_{0}))}{\Leb^{2}( \DD_{r}(z_{0}))}\right)^{\frac{1}{p}}.
$$
We invoke \eqref{z4} as follows. Recall that $u_j$ are $K$-quasiregular and that $u_{j} \to u$ locally uniformly on $\Omega$
and boundedly in $\WW^{1,q}_{\loc}( \Omega )$ for each $q<p$, hence we get as $j \to \infty$:
\begin{eqnarray*}
  \int_{\DD} \! | \D v_{j,r}|^{2} \, \dd m(z) &=& \frac{\pi}{\rho_{r}^2} \fint_{\DD_{r}(z_{0})} \! | \D u_{j}|^{2} \, \dd m(z)
  \leq \frac{\pi K}{\rho_{r}^2}\fint_{\DD_{r}(z_{0})} \! \det \D u_{j} \, \dd m(z)\\
  &\to& \frac{\pi K}{\rho_{r}^{2}} \fint_{\DD_{r}(z_{0})} \! \det \D u \, \dd m(z) \leq
  \frac{\pi K}{\rho_{r}^2} \fint_{\DD_{r}(z_{0})} \! | \D u|^{2} \, \dd m(z)\\
  &\leq& \pi K \left( \rho_{r}^{-p}\fint_{\DD_{r}(z_{0})} \! | \D u|^{p} \, \dd m(z) \right)^{\frac{2}{p}}\\
  &=& \pi K \left( \frac{1}{\lambda \bigl( \DD_{r}(z_{0}) \bigr)}\int_{\DD_{r}(z_{0})} \! | \D u|^{p} \, \dd m(z) \right)^{\frac{2}{p}}.
\end{eqnarray*}
Consequently,
\begin{equation}\label{fromz4}
  \lim_{r \searrow 0} \lim_{j \to \infty} \int_{\DD} \! | \D v_{j,r}|^{2} \, \dd m(z) = 0.
\end{equation}
Next, we get as in Step 2,
$$
\lim_{j \to \infty} \int_{\DD_{r}(z_{0})} \! \B_{p}( \D u_j ) \, \dd m(z) = -\lambda ( \DD_{r}(z_{0}))
$$
when $\lambda ( \partial \DD_{r}(z_{0}))=0$, hence
$$
\lim_{E \notni r \searrow 0}\lim_{j \to \infty} \int_{\DD} \! \B_{p}( \D v_{j,r}) \, \dd m(z) = -1,
$$
where $E \equiv \{ r \in (0,\dist (z_{0},\partial \Omega )) : \, \lambda (\partial \DD_{r}(z_{0})) > 0 \}$ is an at most countable set.
In order to combine this with \eqref{z3} we require some routine arguments from measure theory that we briefly recall here
for convenience of the reader. Define
$$
\Lambda (t) \equiv \limsup_{E \notni r \searrow 0} \frac{\lambda \bigl( \DD_{tr}(z_{0}) \bigr)}{\lambda \bigl( \DD_{r}(z_{0}) \bigr)}, \, t \in (0,1] .
$$
We assert that \eqref{z3} implies that $\Lambda (t) \geq t^{2}$ for all $t \in (0,1]$. Indeed suppose for a contradiction that it fails
at some $s \in (0,1)$, so that $\Lambda (s) < s^2$. First note that since the function $r \mapsto \lambda \bigl( \DD_{r}(z_{0})\bigr)$
is left-continuous and the set $E$ is at most countable we actually have that
$$
\Lambda (t) = \limsup_{r \searrow 0} \frac{\lambda \bigl( \DD_{tr}(z_{0})\bigr)}{\lambda \bigl( \DD_{r}(z_{0}) \bigr)} , \, t \in (0,1] ,
$$
and so choosing $t \in (0,s)$ with $\Lambda (s)< t^2$ we get by iteration that for some $\delta > 0$,
$$
\lambda \bigl( \DD_{s^{j}\delta}(z_{0}) \bigr) < t^{2j}\lambda \bigl( \DD_{\delta}(z_{0})\bigr)
$$
holds for all $j \in \N$. But this contradicts \eqref{z3} since then, as $j \to \infty$,
$$
\frac{\lambda \bigl( \DD_{s^{j}\delta}(z_{0}) \bigr)}{s^{2j}\delta^{2}} < \left( \frac{t}{s} \right)^{2j} \frac{\lambda \bigl( \DD_{\delta}(z_{0})\bigr)}{\delta^{2}} \to 0.
$$
In terms of the maps $v_{j,r}$ this amounts to that the bound
\begin{equation}\label{contra}
\limsup_{E \notni r \searrow 0} \lim_{j \to \infty} \int_{\DD_{s}(0)} \! \B_{p}( \D v_{j,r}) \, \dd m(z) \leq -\pi s^{2}
\end{equation}
holds for each $s \in (0,1]$. In order to conclude the proof we fix $s \in (0,1)$, use \eqref{fromz4} and \eqref{contra} to select a null sequence
$r_{j} \searrow 0$ such that $v_{j} \equiv v_{j,r_{j}} \to 0$ strongly in $\WW^{1,2}( \DD )$ and
$$
\lim_{j \to \infty}\int_{\DD_{s}(0)} \! \B_{p}( \D v_{j}) \, \dd m(z) \leq -\pi s^{2}.
$$
Now $v_j \colon \DD \to \C$ are $K$-quasiregular maps and so we may Stoilow factorize as $v_{j} = h_{j} \circ f_{j}$, where $f_{j}\colon \C \to \C$
are $K$-quasiconformal normalized principal and $h_{j}\colon f_{j}( \DD ) \to \C$ are holomorphic. Taking a further subsequence if necessary (not relabelled)
we may assume that $f_{j} \to f$ uniformly on $\C$, where $f \colon \C \to \C$ is a $K$-quasiconformal normalized principal map. We then also have that
$h_j \to 0$ locally uniformly on $f( \DD )$. It follows that $h_{j}^{\prime} \circ f_{j} \to 0$ uniformly on $\DD_{s}(0)$ and therefore
\begin{eqnarray*}
  -\pi s^{2} &\geq& \lim_{j \to \infty} \int_{\DD_{s}(0)} \! \bigl| h_{j}^{\prime} \circ f_{j}\bigr|^{p}  \B_{p}( \D f_{j}) \, \dd m(z)\\
  &\geq& \sup_{\DD_{s}(0)}\bigl| h_{j}^{\prime} \circ f_{j} \bigr|^{p}\int_{\DD} \! \B_{p}( \D f_{j}) \, \dd m(z)\\
  &\geq& -\pi \sup_{\DD_{s}(0)}\bigl| h_{j}^{\prime} \circ f_{j} \bigr|^{p} \to 0
\end{eqnarray*}
as $j \to \infty$, which is impossible. The set of points $z_{0} \in B$ satisfying \eqref{z3} and \eqref{z4} is therefore empty, and
since $\lambda^s$ almost all points in $B$ have those properties we infer that $\lambda^{s}(B)=0$ as required.
\end{proof}

For the Dirichlet classes we have the following almost immediate corollary.
\begin{corollary}\label{betterBurk}
Let $p>2$ and put $K \equiv p/(p-2)$. Let $\Omega$ be a bounded open subset of $\C$ and $g \colon \C \to \C$ be a $K$-quasiregular map of class $\WW^{1,p}_{\loc}( \C )$
If $u_j \in g +\WW^{1,p}_{0}(\Omega)$ are $K$-quasiregular and $u_{j} \to u$ weakly in $\WW^{1,p}( \Omega )$, then $u \in g+\WW^{1,p}_{0}( \Omega )$ is $K$-quasiregular and
\[
\liminf_{j \to \infty} \int_{\Omega} \! \B_{p}( \D u_j) \, \dd m(z) \geq \int_{\Omega} \! \B_{p}( \D u) \, \dd m(z).
\]
\end{corollary}
We emphasize that we assume $u_j \to u$ weakly in the critical Sobolev space $\WW^{1,p}( \Omega )$ and that this of course automatically entails that $u \in \WW^{1,p}( \Omega )$.
If we only knew that $u_j$, $u$ were $K$-quasiregular and $u_j \to u$ uniformly on $\Omega$, then we would merely have that $\D u$ belonged to the local Marcinkiewicz
space $\LL^{p,\infty}_{\loc}( \Omega )$. Theorem \ref{swlscBurk} does not cover this situation.

\begin{proof}
Consider the metric neighbourhood $\Omega_{r} \equiv \DD_{r}( \Omega ) \equiv \bigl\{ z \in \C : \, \mathrm{dist }(z,\Omega ) < r \bigr\}$ for $r>0$.
It is clearly precompact and $\Omega \Subset \Omega_r$. Since the function $\mathrm{dist }( \cdot , \Omega )$ is Lipschitz we may use the coarea formula to see that
also $\mathscr{L}^{2}( \partial \Omega_r ) =0$ holds for $\mathscr{L}^{1}$ almost all $r>0$, and we fix such an $r$. Now extend the maps $u_j$, $u$ to $\C \setminus \Omega$
by $g$, and record that hereby $u_j$, $u$ are $K$-quasiregular maps in the critical Sobolev space $\WW^{1,p}_{\loc}( \C )$ and that $u_j \to u$ weakly in $\WW^{1,p}_{\loc}( \C )$.
The latter in particular implies that also $u_j \to u$ uniformly on $\C$, and so according to Theorem~\ref{swlscBurk}
$$
\liminf_{j \to \infty}\int_{\Omega_r} \! \B_{p}( \D u_{j}) \, \dd m(z) \geq \int_{\Omega_r} \! \B_{p}( \D u ) \, \dd m(z)
$$
holds. Note that here we have $u_j =u=g$ on $\Omega_{r}\setminus \Omega$ and so by standard properties of Sobolev functions and approximate
derivatives we get in particular that $\D u_j = \D u$ a.e.~on $\Omega_{r}\setminus \Omega$ (also in the pathological case where
$\mathscr{L}^{2}( \partial \Omega )>0$ that has not been excluded). Because $\B_{p}( \D u) \in \LL^{1}_{\loc}( \C )$ the required
conclusion follows from this.
\end{proof}

\section*{Acknowledgments}
D.F, K.A, A.K acknowledge the financial support of QUAMAP, the ERC Advanced Grant 834728, 
 and of  
the Severo Ochoa Programme 
CEX2019-000904-S.
A.G. was supported by Dr.\ Max R\"ossler, the Walter Haefner Foundation and the ETH Z\"urich Foundation.  
D.F and A.K were  partially supported by CM and UAM,
and A.K by Academy of Finland CoE Randomness and Structures, and Academy Fellowship Grant 355840. 
 D.F acknowledges financial support by PI2021-124-195NB-C32. 
  D.F also wants to acknowledge the hospitality of the Mathematical Institute of Oxford during the Summer of 2023.

\end{document}